\documentclass{article}
\usepackage{amsmath}
\usepackage{amsthm}
\usepackage{amsfonts}
\usepackage{amssymb}
\usepackage{upgreek}
\usepackage{graphicx}
\usepackage{hyperref}
\usepackage{enumerate}
\usepackage{enumitem}
\usepackage{tikz-cd}
\usepackage{array}
\usepackage[parfill]{parskip}
\usepackage{subcaption}

\theoremstyle{definition}
\newtheorem{thm}{Theorem}[section]
\newtheorem{defn}[thm]{Definition}
\newtheorem{defn-prop}[thm]{Definition-Proposition}
\newtheorem{rem}[thm]{Remark}
\newtheorem{lem}[thm]{Lemma}
\newtheorem{exmp}[thm]{Example}

\newtheorem{prop}[thm]{Proposition}
\newtheorem{notat}[thm]{Notation}
\newtheorem{conj}[thm]{Conjecture}

\usepackage[fontsize=10pt]{scrextend}
\usepackage{setspace}
\usepackage[utf8]{inputenc}
\usepackage[
left=2.45cm,
right=2.45cm,
top=2.3cm,
bottom=2.3cm,
]{geometry}

\DeclareMathOperator{\grad}{grad}

\DeclareMathOperator{\U}{\text{U}}
\newcommand{\red}{\text{red}}
\newcommand{\HMR}{\textit{HMR}}
\newcommand{\HM}{\textit{HM}}
\newcommand{\HF}{\textit{HF}}
\newcommand{\Kh}{\textit{Kh}}
\newcommand{\CKh}{\textit{CKh}}
\newcommand{\Khr}{\textit{Khr}}
\newcommand{\CKhr}{\textit{CKhr}}

\renewcommand{\Im}{\text{Im}}

\DeclareMathOperator{\fraka}{\mathfrak{a}}
\DeclareMathOperator{\frakb}{\mathfrak{b}}
\DeclareMathOperator{\frakc}{\mathfrak{c}}
\DeclareMathOperator{\frakd}{\mathfrak{d}}

\DeclareMathOperator{\fraks}{\mathfrak{s}}

\DeclareMathOperator{\frakt}{\mathfrak{t}}

\newcommand{\calB}{\mathcal{B}}

\newcommand{\calG}{\mathcal{G}}

\newcommand{\calL}{\mathcal{L}}

\newcommand{\bbZ}{\mathbb{Z}}
\newcommand{\bbB}{\mathbb{B}}
\newcommand{\bbN}{\mathbb{N}}
\newcommand{\bbF}{\mathbb{F}}
\newcommand{\bbC}{\mathbb{C}}

\newcommand{\bbT}{\mathbb{T}}

\newcommand{\bbY}{\mathbb{Y}}
\newcommand{\bbR}{\mathbb{R}}

\newcommand{\bfC}{\mathbf{C}}
\newcommand{\bfF}{\mathbf{F}}
\newcommand{\bfzero}{\mathbf{0}}
\newcommand{\bfone}{\mathbf{1}}
\newcommand{\bftwo}{\mathbf{2}}
\newcommand{\bbW}{\mathbb{W}}

\newcommand{\reals}{\mathbb R}

\newcommand{\del}{\ensuremath{\partial}}
\newcommand{\delbar}{\ensuremath{{\bar{\partial}}}}

\newcommand{\spinc}{{\text{spin\textsuperscript{c} }}}

\newcommand{\qgrade}{{\textbf{q}}}
\newcommand{\pgrade}{{\textbf{p}}}
\newcommand{\hgrade}{{\textbf{h}}}

\newcommand{\dbcv}{\mathsf{\Sigma}}

\newcommand{\Nbhd}{\ensuremath{{\sf{Nbhd}}}}
\newcommand{\gr}{\textsf{gr}}
\newcommand{\Fltr}{\textsf{Fltr}}

\newcommand{\la}{\ensuremath{{\langle}}}

\newcommand{\ra}{\ensuremath{{\rangle}}}

\makeatletter
\DeclareRobustCommand\widecheck[1]{{\mathpalette\@widecheck{#1}}}
\def\@widecheck#1#2{%
    \setbox\z@\hbox{\m@th$#1#2$}%
    \setbox\tw@\hbox{\m@th$#1%
       \widehat{%
          \vrule\@width\z@\@height\ht\z@
          \vrule\@height\z@\@width\wd\z@}$}%
    \dp\tw@-\ht\z@
    \@tempdima\ht\z@ \advance\@tempdima2\ht\tw@ \divide\@tempdima\thr@@
    \setbox\tw@\hbox{%
       \raise\@tempdima\hbox{\scalebox{1}[-1]{\lower\@tempdima\box
\tw@}}}%
    {\ooalign{\box\tw@ \cr \box\z@}}}
\makeatother

\newcommand{\Addresses}{{
  \bigskip
  \footnotesize

\textsc{Department of Mathematics, 
Harvard University, 
Cambridge MA, 
United States 02138}\par\nopagebreak
  \textit{E-mail address:} \texttt{jiakaili@math.harvard.edu}
}}

\newcommand\blfootnote[1]{%
	\begingroup
	\renewcommand\thefootnote{}\footnote{#1}%
	\addtocounter{footnote}{-1}%
	\endgroup
}

\title{Real monopoles and a spectral sequence from Khovanov homology}
\begin{document}    
\author{Jiakai Li}
\maketitle
\begin{abstract}
Given a based link $(K,p)$, we define a ``tilde''-version $\widetilde{\HMR}(K,p)$ of real monopole Floer homology and prove an unoriented skein exact triangle.
We show the Euler characteristic of $\widetilde{\HMR}(K,p)$ is equal to Miyazawa's~\cite{Miyazawa2023} invariant $|\deg(K)|$ and examine some examples.
Further, we construct a spectral sequence over $\bbF_2$ abutting to $\widetilde{\HMR}(K,p)$, whose $E_2$ page is the reduced Khovanov homology $\Khr(\overline{K})$ of the mirror link $\overline{K}$.
\end{abstract}
\section{Introduction}
\subsection{Structural results}
Given a link $K$ in $S^3$, the author~\cite{ljk2022} defined three \emph{real monopole Floer homology groups}
$\widehat{\HMR}_{\bullet}(K)$,
	$\widecheck{\HMR}_{\bullet}(K)$,
	$\overline{\HMR}_{\bullet}(K)$.
We write  $\HMR^{\circ}_{\bullet}(K)$ where $\circ \in \{\wedge,\vee,-\}$ is a placeholder for the flavour.
These groups are constructed as analogues to Kronheimer-Mrowka's~\cite{KMbook2007} monopole Floer homology $\HM^{\circ}_{\bullet}(Y)$ of a $3$-manifold $Y$.
The counterpart of the $(-2)$-degree monopole Floer ``$U_{\dag}$-map'' in $\HMR^{\circ}_{\bullet}(K)$ is a $(-1)$-degree map
\[
	\upsilon_p:\HMR^{\circ}_{\bullet}(K) \to \HMR^{\circ}_{\bullet}(K),
\]
where $p \in K$.
The operator $\upsilon_p$ depends only on the component of $K$ on which $p$ lies.
\blfootnote{This work was partially supported by a Simons Foundation Award \#994330 (Simons Collaboration on New Structures in Low-Dimensional Topology).}

The first goal of this paper is to introduce the following ``tilde''-version of $\HMR$ of a based link $(K,p)$ as the mapping cone of $\upsilon_p$:
\[
\widetilde{\HMR}_{\bullet}(K,p),
\]
which is a finite $\bbF_2$-vector space.
This algebraic construction is due to Bloom~\cite{Bloom2011} in the case of $\widetilde{\HM}_{\bullet}(Y)$, in analogy with $\widehat{\textit{HF}}(Y)$ in Heegaard Floer homology~\cite{OzSz2004}.

One should compare $\HMR_{\bullet}^{\circ}(K)$ with a version (Heegaard, monopole, instanton etc.) of Floer homology of the double branched cover $\dbcv_2(S^3,K)$.
Such a functor is manifestly a link invariant, and has been studied extensively in the literature.
Notably, several spectral sequences from Khovanov homology to a Floer homology of double branched covers were discovered, since the work of Osv\'ath and Szab\'o~\cite{OzSz2005HFdc} in the Heegaard Floer setting.
The most relevant result for us is Bloom's spectral sequence~\cite{Bloom2011} from the Khovanov homology $\Kh(K)$ to $\widetilde{\HM}_{\bullet}(\dbcv_2(S^3,\overline{K}))$.
By adapting Bloom's arguments, we prove the following.
\begin{thm}
	\label{thm:SS_intro}
There is a spectral sequence $\{(E_i(K,p),d_i)\}_{i\ge 1}$ abutting to $\widetilde{\HMR}(K,p)$, whose $E_2$-page is the reduced Khovanov homology $\Khr(\overline{K})$ of the mirror $\overline{K}$ over $\bbF_2$.
Each page $E_i(K,p)$ is an invariant of the based link $(K,p)$.
\end{thm}

Along the way, we prove an exact triangle for a triple of links  satisfying the following local \emph{unoriented skein relation}, away from a point $p$:
\begin{figure}[tbh]
	\centering
\tikzset{every picture/.style={line width=0.75pt}} 
\begin{tikzpicture}[x=0.28pt,y=0.28pt,yscale=-1,xscale=1]
	\draw   (11.33,102.56) .. controls (11.33,52.17) and (52.57,11.33) .. (103.44,11.33) .. controls (154.32,11.33) and (195.56,52.17) .. (195.56,102.56) .. controls (195.56,152.94) and (154.32,193.78) .. (103.44,193.78) .. controls (52.57,193.78) and (11.33,152.94) .. (11.33,102.56) -- cycle ;
	\draw   (235.33,101.89) .. controls (235.33,51.51) and (276.57,10.67) .. (327.44,10.67) .. controls (378.32,10.67) and (419.56,51.51) .. (419.56,101.89) .. controls (419.56,152.27) and (378.32,193.11) .. (327.44,193.11) .. controls (276.57,193.11) and (235.33,152.27) .. (235.33,101.89) -- cycle ;
	\draw   (465.33,102.56) .. controls (465.33,52.17) and (506.57,11.33) .. (557.44,11.33) .. controls (608.32,11.33) and (649.56,52.17) .. (649.56,102.56) .. controls (649.56,152.94) and (608.32,193.78) .. (557.44,193.78) .. controls (506.57,193.78) and (465.33,152.94) .. (465.33,102.56) -- cycle ;
	\draw    (39.33,37.67) -- (166.89,167.11) ;
	\draw    (94.89,105.78) -- (40.22,167.11) ;
	\draw    (162.22,33.11) -- (107.56,94.44) ;
	\draw    (272.22,30.44) .. controls (328.22,71.11) and (287.56,177.11) .. (268.22,172.44) ;
	\draw    (375.56,24.44) .. controls (320.89,75.11) and (362.22,156.44) .. (387.56,171.78) ;
	\draw    (482,52) .. controls (518.22,93.11) and (597.56,94.44) .. (628.89,45.78) ;
	\draw    (484.22,157.78) .. controls (526.44,119.56) and (568.22,117.78) .. (632.89,154.44) ;
	
	\draw (87,202.07) node [anchor=north west][inner sep=0.75pt]    {$K_{2}$};
	\draw (305.33,202.07) node [anchor=north west][inner sep=0.75pt]    {$K_{1}$};
	\draw (540.33,202.07) node [anchor=north west][inner sep=0.75pt]    {$K_{0}$};
\end{tikzpicture}.
\caption{Three links satisfying an unoriented skein relation.}
\label{fig:triangle210_intro}
\end{figure}
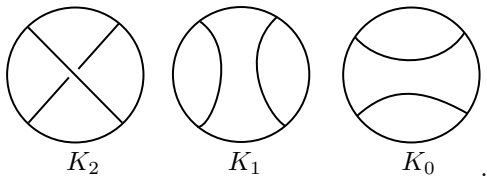
\begin{thm}
	\label{thm:triangle_intro}
Let $(K_0,K_1,K_2)$ be a triple of links satisfying the unoriented skein relation. Then there is an exact triangle
\begin{equation}
\label{eqn:triangle}
	\begin{tikzcd}
	{\widetilde{\HMR}_{\bullet}(K_2,p)} && {\widetilde{\HMR}_{\bullet}(K_1,p)} \\
	& {\widetilde{\HMR}_{\bullet}(K_0,p)}
	\arrow[from=1-1, to=1-3]
	\arrow[from=1-3, to=2-2]
	\arrow[from=2-2, to=1-1]
\end{tikzcd}.
\end{equation}
\end{thm} 
Unlike Floer homologies of branched double covers, the construction of $\widetilde{\HMR}$ depends on the covering involutions.
The proofs of Theorem~\ref{thm:triangle_intro} and Theorem~\ref{thm:SS_intro}, building upon~\cite[Theorem~1.1]{ljk2023triangle}, rely ultimately  on the reducibles, which are preserved under the involutions.

\subsection{Observations and remarks}
\label{subsec:intro_obs}
The salient feature of $\HMR^{\circ}_{\bullet}(K)$ is its grading.
By restricting to invariant subspaces under anti-linear involutions, we introduce factors of $\frac12$ into various index formulae.
This is related to the basic fact that a real subspace $V^{\reals}$ of a complex vector space $V$ has dimension $\frac{1}{2}\dim_{\reals}V$.
As a result, 
in the case of non-torsion \spinc structures, $\HMR^{\circ}_{\bullet}(K)$ 
loses the absolute $\bbZ/2$ grading that is available in ordinary monopole Floer homologies.
Similarly, the induced map from a given cobordism no longer has a homogenous $\bbZ/2$-degree among all \spinc structures.

To be precise, $\widetilde{\HMR}_{\bullet}(K,p)$ is a direct  sum over all real \spinc structures over $\dbcv_2(K)$.
A real \spinc structure $(\fraks,\uptau)$ consists of a \spinc structure $\fraks$ on $\dbcv_2(K)$ and an anti-linear involution $\uptau$ on the spinor bundle lifting the covering involution.
If $\fraks$ is torsion, then $\widetilde{\HMR}_{\bullet}(K,p,\fraks,\uptau)$ admits a relative $\bbZ$-grading, much like $\widetilde{\HM}_{\bullet}(\dbcv_2(K),\fraks)$.
The aforementioned factor of $\frac12$ changes the computation of the Euler characteristics of $\widetilde{\HMR}_{\bullet}$, which is now well-defined up to an overall sign for the lack of an absolute $\bbZ/2$-grading.
We demonstrate this phenomenon as follows.
Denote
\[
	|\chi_{\HMR}|(K,p,\fraks,\uptau) = \left|\chi(\widetilde{\HMR}_{\bullet}(K,p,\fraks,\uptau))\right|.
\]
We will exhibit a knot $K$ for which $|\chi_{\HMR}|(K,p)$ and $|\chi(\widetilde{\HM}(\dbcv_2(K)))|$ disagree, even though there is a natural $1$-to-$1$ correspondence between the elements of $\widetilde{\HMR}_{\bullet}(K,p,\fraks,\uptau)$ and $\widetilde{\HM}_{\bullet}(\dbcv_2(K),\fraks)$.
Note that there is no such correspondence for an arbitrary link.
Let us suppress the base point $p$ in notations if $\widetilde{\HMR}_{\bullet}(K,p,\fraks,\uptau)$ is independent of $p$.

Let $K$ be the pretzel knot $P(-2,3,7)$.
Manolescu pointed out to the author that $|\chi_{\HMR}|(K)$ is equal to $3$ at its unique real \spinc structure.
We elaborate this fact in Section~\ref{subsec:euler_char}. 
On the other hand, the double branched cover of $K$ is the Brieskorn sphere $-\Sigma(2,3,7)$ with the  opposite orientation.
The generators of the chain complex of  $\widetilde{\HMR}_{\bullet}(K,\fraks,\uptau)$ are in bijection with the generators of the chain complex of $\widetilde{\HM}_{\bullet}(-\Sigma(2,3,7))$. 
Nevertheless, this correspondence does not preserve gradings. 
In Example~\ref{exmp:237} we will complete the calculations that
\[
\left|\chi(\widetilde{\HM}_{\bullet}(-\Sigma(2,3,7)))\right| = 1 \ne 3 = \big|\chi_{\HMR}\big|(P(-2,3,7)).
\]

Roughly speaking, $\widetilde{\HMR}_{\bullet}$ can be seen as a formal analogue of the $\bbF_2$-homology of the framed $3$-dimensional real Seiberg-Witten moduli space.
Meanwhile, Miyazawa~\cite[Definition~4.28]{Miyazawa2023} introduced an invariant for a nonzero determinant link:
\[|\text{deg}(K)|\]
as the absolute value of the signed count of the (framed) real Seiberg-Witten solutions on $\dbcv_2(K)$ with a real \spinc structure $(\fraks,\uptau)$.
The invariants $|\chi_{\HMR}|(K)$ and $|\deg(K)|$ coincide for $P(-2,3,7)$.
In fact, we prove the two invariants agree:
\begin{thm}
\label{conj:chi=deg}	
Given a nonzero determinant link $K$ and a real spin structure $(\fraks,\uptau)$ on $\dbcv_2(K)$, we have 
\[
|\chi_{\HMR}|(K,\fraks,\uptau) = |\deg(K)|.
\] 
\end{thm}
Notice that the left hand side is defined algebraically as a mapping cone while the right hand side is defined directly using framed moduli spaces of real Seiberg-Witten solutions.
Recently,
S. Kang, J. Park, and M. Taniguchi~\cite{KPT2024} computed more examples of $|\deg(K)|$ using Dai-Stroffregen-Sasahira's lattice homotopy type~\cite{DSM2023}, together with~\cite[Proposition~5.2]{KPT2024}:
\[
	|\deg(K)|=|\chi(\textit{SWF}_R(K))|,
\]
where $\textit{SWF}_R(K)$ is the real Seiberg-Witten Floer spectrum of Konno-Miyazawa-Taniguchi~\cite{KMT2023}.
We conjecture the following isomorphism.
\begin{conj}
\[
\widetilde{\HMR}_{\bullet}(K) \cong \tilde{H}_*(\textit{SWF}_R(K);\bbF_2).
\]
\end{conj}
See also \cite[Conjecture~1.28]{KMT2023}.
The isomorphism is in $\bbF_2$-coefficient, but we hope to define $\HMR^{\circ}_{\bullet}$ over $\bbZ$-coefficient in subsequent work and expect the isomorphism to hold over $\bbZ$ also.

It is currently unknown whether $|\chi_{\HMR}|(K)$ can be expressed in terms of classical link invariants.
Despite the exact triangle~\eqref{eqn:triangle}, it is not straightforward to write down a relation, if exists, for $|\chi_{\HMR}|(K_i,p)$ of an unoriented skein triple of links $(K_2,K_1,K_0)$.
This complication arises from the lack of absolute gradings, and the fact that each of the maps in~\eqref{eqn:triangle} is a sum of homomorphisms that are not supported on a single $\bbZ/2$-grading.
Computing changes of $|\chi_{\HMR}|$ in a skein relation requires further understanding of the cobordism maps involved and possibly groups without relative $\bbZ/2$-gradings.
We will give an example of a skein exact triangle starting from $P(-2,3,7)$ in Remark~\ref{rem:noskein}

Back to the spectral sequence $\{(E_i(K,p),d_i)\}_{i\ge 1}$ abutting to $\widetilde{\HMR}(K,p)$.
The intermediate pages between $E_2$ and $E_{\infty}$ should produce new link invariants.
It is known that $\widetilde{\HMR}_{\bullet}(K)$ differs from existing Floer homologies admitting spectral sequences from the Khovanov homology such as $\widehat{\HF}(\dbcv_2(K))$~\cite{OzSz2005HFdc}, $I^{\sharp}(K)$~\cite{KMunknot2011}, or $\widehat{\textit{HFK}}(K)$~\cite{Dowlin2024SSHFK}.
Indeed, $\widetilde{\HMR}_{\bullet}(K) = \bbF_2$ for any torus knot, whereas the Khovanov homology and the aforementioned homologies are highly nontrivial.
It follows that there are far more differentials in the $\Khr \Rightarrow \widetilde{\HMR}_{\bullet}$ spectral sequence than the above spectral sequences.

The author expects the spectral sequence in Theorem~\ref{thm:SS_intro} to fit into the framework of \emph{Khovanov-Floer homology} by Baldwin-Hedden-Lobb~\cite{BHL2019}.
In particular, every page of $E_i$ should be functorial with respect to based link cobordisms.

\subsection{Organization of sections}
We define $\widetilde{\HMR}_{\bullet}(K,p)$ in Section~\ref{sec:HMRtildedefn} and make some remarks about its Euler characteristic.
We review the definition of Khovanov homology in Section~\ref{subsec:odd_Kh}.
In Section~\ref{sec:a_SS}, we construct the spectral sequence in Theorem~\ref{thm:SS_intro} and prove both Theorem~\ref{thm:SS_intro} and Theorem~\ref{thm:triangle_intro}.

\section{$\widetilde{\HMR}_{\bullet}(K,p)$}
\label{sec:HMRtildedefn}
For the scope of this paper, the links of our interests will always lie in $S^3$, and the corresponding cobordisms will be properly embedded surfaces of $S^4$ with some standard $4$-balls removed.
It is customary to write, for instance, $K$ for a link $K \subset S^3$ and $\Sigma:K_1 \to K_0$ for a properly embedded surface $\Sigma \subset I \times S^3$ whose boundary is $\overline{K}_0 \sqcup K_1$, where $I$ is an interval.

To emphasize the ambient $3$- and $4$-manifolds, we adopt the following conventions.
We denote by $(Y,K)$ a pair of a $3$-manifold with a link $K \subset Y$.
Alternatively, $(Y,K)$ is viewed as an orbifold having cone angle $\pi$ along $K$.
Blackboard letters are reserved for double branched covers.
For example, $\bbY$ represents the double branched cover of $Y$ along $K$.
Similarly, a cobordism from $(Y_1,K_1) \to (Y_0,K_0)$ is written as $(W,\Sigma)$, where $\Sigma$ is a properly embedded surface in a $4$-manifold $W$.
The branched double cover of $(W,\Sigma)$ is $\bbW$ and there is a deck transformation $\upiota_W: \bbW \to \bbW$.

Of course, in this paper typically $Y = S^3$ and $W = S^4 \setminus \cup_{1 \le i \le n} B^4$. 
We remind the readers that the author's previous construction \cite{ljk2022} extends readily to any links $K$ in any $3$-manifold $Y$, though the exposition there focuses on $Y=S^3$.
\subsection{Review of ${\HMR}^{\circ}_{\bullet}(K)$}
\label{sec:review_hmr}
For more details, see \cite{ljk2022}.
Given a link $K \subset S^3$, let $\bbY = \dbcv_2(S^3,K)$ be its branched double cover and let $\upiota:\bbY \to \bbY$ be the deck transformaton.
Fix an $\upiota$-invariant Riemannian metric on $\bbY$.
In \cite[Definition~3.7]{ljk2022}, the author defined a \emph{real \spinc structure} over $(\bbY,\upiota)$ to be a pair $(\fraks, \uptau)$ of a \spinc structure $\fraks = (S,\rho)$ on $\bbY$ and an anti-linear involutive lift $\uptau:S \to S$ of $\upiota$ on the spinor bundle, compatible with the Clifford multiplication $\rho$.
A real \spinc structure $(\fraks_{\Sigma},\uptau_{\Sigma})$ on a branched double cover of a surface $\Sigma$ is defined analogously.
As a final note, real \spinc structures are essential inputs of our invariants, for double branched covers of $S^3$ or $S^4 \setminus (\cup_n B^4)$, every \spinc structure admits a unique real structure so it suffices to specify the underlying \spinc structure.

The real monopole Floer homology group $\HMR^{\circ}_{\bullet}$ is an $\infty$-dimensional analogue of equivariant Morse homology of the Chern-Simons-Dirac functional $\calL$ on the Seiberg-Witten configuration space $\calB(K,\fraks,\uptau)$.
The equivariance is due to the nonfree action of $(\pm 1)$ constant gauge transformations.
In practice, one considers the blown-up configuration space $\calB^{\sigma}(K,\fraks,\uptau)$, views it as a manifold with boundary, and studies trajectories of a blown-up version of $\grad \calL$.
Furthermore, one perturbs the gradient $\grad \calL$ abstractly (chosen from a large Banach space of perturbations) to ensure all the relevant moduli spaces of are regular.

There are three types of critical points of $\grad \calL$ on the blown-up configuration space: \emph{interior} (irreducible), \emph{boundary-stable} (reducible), and \emph{boundary unstable} (reducible).
Let $C^o,C^s,C^u$ denote the $\bbF_2$-vector space   generated by interior, boundary-stable, and boundary-unstable critical points, respectively.
We define three Floer homology groups
\[
	\widehat{\HMR}_{*}(K,\fraks,\uptau), \quad 
	\widecheck{\HMR}_{*}(K,\fraks,\uptau), \quad 
	\overline{\HMR}_{*}(K,\fraks,\uptau)
\]
using a combination of interior and boundary critical points and counting trajectories between them.
For example, the ``\HMR-to'' complex is given by
	\begin{equation*}
		\check C(K,\fraks,\uptau) = C^o \oplus C^s, \quad \check\del = 
		\begin{pmatrix}
			\del^o_o && -\del^u_o\bar\del^s_u\\
			\del^o_s && \bar\del^s_s - \del^u_s\bar\del^s_u
		\end{pmatrix},
\end{equation*}	 
where $\del^o_s$ counts index-$1$ trajectories of $\grad\calL$  from interior critical points to the boundary-stable ones,  $\overline{\del}^s_u$ counts reducible trajectories from boundary-stable critical points to the boundary-unstable ones, et cetera.

The groups above are graded, and we complete the grading ``$*$'' in the negative direction to define cobordisms maps, indicated by ``$\bullet$''.
Whenever the real \spinc structure $(\fraks,\uptau)$ is dropped from notations, we take the sum over all real \spinc structures.

Given a cobordism $\Sigma:K_1 \to K_0$, let $\bbW$ be the double branched cover $\dbcv_2(I \times S^3, \Sigma)$, treated as a $4$-manifold with involution.
Suppose $(\fraks_{\Sigma},\uptau_{\Sigma})$ is a real \spinc structure on $(\bbW,\upiota_W)$.
To each flavour $\circ \in \{\wedge,\vee,-\}$ we have a homomorphism
\[
	\HMR^{\circ}(\Sigma,\fraks_{\Sigma},\uptau_{\Sigma}):
	\HMR^{\circ}_{\bullet}(K_1,\fraks_1,\uptau_1) \to 
	\HMR^{\circ}_{\bullet}(K_0,\fraks_0,\uptau_0),
\]
where $(\fraks_i,\uptau_i)$ are restrictions of real \spinc structures on the boundaries.
Each of these maps is defined on the chain level by counting the \emph{moduli space}
\[M_z(\fraka,\Sigma^*,\frakb)\]
\emph{of real Seiberg-Witten solutions} on $(\bbW^*,\upiota_W)$, where the superscript ``$*$'' indicates that we are adding two cylindrical ends 
\[(-\infty,0] \times \dbcv_2(K_1), \quad \text{and} \quad
[1,\infty) \times \dbcv_2(K_0)\] 
to the respective boundary components of $\bbW$ at $\{0\}$ and $\{1\}$. 
Alternatively, $\bbW^*$ is the branched double cover of the surface $\Sigma^*$ obtained from $\Sigma$ by adjoining  cylindrical ends, embedded in $\reals \times S^3$.
An element of $M_z(\fraka,\Sigma^*,\frakb)$ is required to be asymptotic to critical points $\fraka$ and $\frakb$ on $ \dbcv_2(K_1)$ and $\dbcv_2(K_0)$ at $-\infty$ and $+\infty$, respectively.
The subscript ``$z$'' denotes a homotopy class of the moduli space.
Implicitly, we choose regular perturbations supported near the neck of the boundaries of $\bbW$ so that all $M_z(\fraka,\Sigma^*,\frakb)$ are transverse.

More generally, let $\calB^{\sigma}(\Sigma,\fraks_{\Sigma},\uptau_{\Sigma})$ be the $4$-dimensional real Seiberg-Witten configuration space on $(\bbW,\upiota_W)$ and let $\calG(\Sigma)$ be the gauge group.
Given a cohomology class $u \in H^*(\calB^{\sigma}(\Sigma,\fraks_{\Sigma},\uptau_{\Sigma});\bbF_2)$, 
there is a homomorphism 
\[
	{\HMR}^{\circ}(u|\Sigma,\fraks_{\Sigma},\uptau_{\Sigma}): 
	{\HMR}_{\bullet}^{\circ}(K_1,\fraks_1,\uptau_1) \to {\HMR}_{\bullet}^{\circ}(K_0,\fraks_0,\uptau_0).
\]
that simultaneously evaluates the class $u$ over the relevant Moduli spaces.
As a special case, if $\Sigma = [0,1] \times K$, the evaluation map $\calG(\Sigma) \to \bbZ/2$ at $p \in K$:
\[
	g \mapsto g(p)
\]
defines an element $\upsilon_p \in H^1(\calB^{\sigma}(\Sigma,\fraks_{\Sigma},\uptau_{\Sigma}))$.
The class $\upsilon_p$ depends only the component of $K$ on which $p$ lies and defines a module structure on $\HMR^{\circ}_{\bullet}(K)$.
If $K$ has $n$ components, then $\HMR^{\circ}_{\bullet}(K)$ is a module over the ring
\[
	\bbF_2[\upsilon_1,\dots,\upsilon_n]/(\upsilon_i^2=\upsilon_j^2).
\]
Moreover, suppose $\bbT = H^1(\bbY;i\reals)/2\pi H^1(\bbY;i\bbZ)$. Then $\HMR^{\circ}_{\bullet}(K)$ is also a module of $H^*(\bbT;\bbF_2) \cong \Lambda^*[x_1,\dots,x_{\ell}]$ where $\Lambda^*$ denotes the exterior algebra and $\ell$ is equal to $b_1(\bbY)$.
\subsection{Definition of $\widetilde{\HMR}_{\bullet}(K,p)$}
The functor $\widetilde{\HMR}_{\bullet}(K,p)$ is more appropriately defined on the category $\textbf{Link}_{\infty}$ of based links.
An object in $\textbf{Link}_{\infty}$  is a link $K$ in $S^3$ containing a based point $p=\infty \in S^3$.
A morphism in $\textbf{Link}_{\infty}$  is an isotopy class of smoothly properly embedded surface $\Sigma$ in $[0,1] \times S^3$ that contains $[0,1] \times \{\infty\}$.
The isotopies are required to fix pointwise a neighbourhood of the boundary and $[0,1] \times \{\infty\}$.
In what follows, our links will be based and the $\upsilon$-map will be understood as taken over $ \infty$. 

Let $\upsilon$ be the first cohomology class obtained by evaluating gauge transformations at $\{\frac14\} \times \{\infty\}$.
We reinterpret $\widecheck{\HMR}(\upsilon|\Sigma)$ as in Bloom~\cite{Bloom2011}.
Removing a small $4$-ball $B_o$ from $([0,1] \times S^3, \Sigma)$ around $\{\frac14\} \times \{\infty\}$, we obtain a surface in $([0,1] \times S^3)\setminus B_o$ with an additional unknot boundary $\U_1 \subset \del  B_o$.
We regard $(\del B_o, \U_1)$ as a new incoming end.
Now, to this new pair we attach three infinite cylindrical ends \[
\bigg((-\infty,0) \times S^3, (-\infty,0) \times K_1
\bigg), 
\ 
\bigg((1,+\infty) \times S^3, (1,+\infty) \times K_0
\bigg), \
\bigg((0,\infty) \times \del B_o, (0,\infty) \times \U_1
\bigg)\] 
to obtain $(W^{o*},\Sigma^{o*})$.
Let $\bbW^{o*}$ be the double branched cover of $W^{o*}$ along $\Sigma^{o*}$ equipped with deck transformation $\upiota_W$.
We choose an $\upiota_W$-invariant metric, real \spinc structure, and perturbation so that on the $(\del B_o,\U_1)$ boundary, the critical points for $\U_1$ is generated a tower of reducibles $\{\frakc_i: i \in \bbZ\}$.
(They are in one-to-one correspondence with the spectrum of the Dirac operator, where $\frakc_0$  lowest positive eigenvalue.)
For this choice, we have $\gr(\frakc_i) = i$, and there are zero trajectories from $\frakc_i$ to $\frakc_{i-1}$ modulo two.

We define \[\check m(\upsilon \ | \ \Sigma): C^o(K_1) \oplus C^s(K_1) \to C^o(K_0) \oplus C^s(K_0)\] by its entries $m^{**}_*,\bar m^{**}_*$ which counts the zero-dimensional pieces of the monopole moduli space $M_z(\fraka, \frakc_{-2}, \Sigma^{**}, \frakb)$, where $\fraka,\frakb$ are critical points of various types, over $K_1,K_0$ respectively: :
\[
	\check m(\upsilon \ | \ \Sigma) =
	\begin{bmatrix}
		m^{uo}_o(\frakc_{-2} \otimes \ \cdot ) &
		m^{uu}_o(\frakc_{-2} \otimes \delbar^s_u (\cdot) ) + \del^u_o\bar m^{us}_u(\frakc_{-2} \otimes \ \cdot )\\
		m^{uo}_s(\frakc_{-2} \otimes \ \cdot ) &
		\bar m^{uu}_s(\frakc_{-2} \otimes \delbar^s_u (\cdot) ) +
		m^{uu}_o(\frakc_{-2} \otimes \delbar^s_u (\cdot)) +
		\del^u_s m^{us}_u(\frakc_{-2} \otimes \ \cdot )
	\end{bmatrix}.
\]
Here $*$ is a placeholder for $o$ (interior), $s$ (boundary-stable), or $u$ (boundary-unstable).
The map $\check m(\upsilon|\Sigma)$ can be seen to be chain homotopic to \cite[Proposition~13.11]{ljk2022}, using the fact that there are even number of isolated trajectories between $\frakc_i$ and $\frakc_{i-1}$ (in the ordinary case, there does not exists any trajectory for index reason). 
\begin{defn}
The \emph{$\widetilde{\HMR}$ chain complex} $\widetilde C(K)$  is given by the mapping cone of $\check m(\upsilon \ | \ [0,1] \times K)$, consisting of two copies of the ``to''-complexes, where
\[
\widetilde C(K) = \widecheck C(K) \oplus \widecheck C(K), \quad
\widetilde\del
= \begin{bmatrix}
	\check\del & 0\\
	\check m(\upsilon \ | \ [0,1] \times K) & \check\del
\end{bmatrix}.
\]
\end{defn}
There is a long exact sequence
\[
\begin{tikzcd}
	\cdots \ar[r,"j"] &
	\widetilde{\HMR}_{*}(K) \ar[r,"i"] &
	\widecheck{\HMR}_{*}(K) \ar[r,"\upsilon"] &
	\widecheck{\HMR}_{*}(K) \ar[r,"j"] &
	\cdots
\end{tikzcd}
\]
The definition of cobordism map in $\widetilde{\HMR}$ involves counting Seiberg-Witten monopoles over a $1$-parameter family of metrics.
The setup is slightly different from \cite{Bloom2011}.
We begin with the orbifold picture.
Let $(W^o,\Sigma^o)$ be 
\[(([0,1] \times S^3, \Sigma) \setminus B_o, \Sigma \setminus B_o)\] 
and $(Y_{11}, K_{11}), (Y_{00},K_{00})$ be the boundaries $\{0\} \times (S^3,K_1), \{1\} \times (S^3,K_0)$, respectively.

Let $\epsilon > 0$ be small on that $\Sigma^o \cap [0,\epsilon] \times S^3$ is a product.
Consider the ``pushed-in hypersurface'' $( Y_{10}, K_{10})$ in $( W^o,\Sigma^o)$, formed by two pieces. 
The first piece of $ Y_{10}$ is 
$\{\epsilon\} \times (S^3 \setminus \Nbhd(\infty))$, where $\Nbhd(\infty) \subset S^3$ is a small ball around $\infty$.
The second piece is a  neighbourhood of the straight arc $[0,1/4] \times \{\infty\}$, joined to $\{\epsilon\} \times (S^3 \setminus \Nbhd(\infty))$ by smoothing the corner.
The $3$-dimensional hypersurface $ Y_{01}$ intersects $\Sigma^o$ at a circle $K_{01}$ (i.e. a hypersurface in $\Sigma^o$), which consists of $\{\epsilon\} \times K_0$ and an arc that goes around the circle boundary around $\{1/2\} \times \{\infty\}$.
In particular, the circle is contained in the region bounded by $K_{00}$ and $K_{01}$.
Similarly, we construct $(Y_{10},K_{10})$ as the ``pushed-in hypersurface'' of $(Y_{11},K_{11})$ at $(1-\epsilon)$.
Let $ W_{00,10}$ be the cobordism between $Y_{00}$ and $Y_{10}$, and $ W_{01,11}$ be the cobordism between $Y_{01}$ and $Y_{11}$.

\begin{figure}[ht]
\centering
\tikzset{every picture/.style={line width=0.75pt}} 
\begin{tikzpicture}[x=0.6pt,y=0.6pt,yscale=-1,xscale=1]

\draw    (20,37) -- (630,38) ;
\draw    (20,37) -- (20,267) ;
\draw    (20,267) -- (630,267) ;
\draw    (630,38) -- (630,267) ;
\draw    (70,37) -- (70,157) ;
\draw   (330,202) .. controls (330,199.24) and (332.24,197) .. (335,197) .. controls (337.76,197) and (340,199.24) .. (340,202) .. controls (340,204.76) and (337.76,207) .. (335,207) .. controls (332.24,207) and (330,204.76) .. (330,202) -- cycle ;
\draw    (70,157) -- (350,157) ;
\draw    (70,247) -- (350,247) ;
\draw    (70,247) -- (70,267) ;
\draw    (570,37) -- (570,177) ;
\draw    (300,177) -- (570,177) ;
\draw    (300,227) -- (570,227) ;
\draw    (570,227) -- (570,267) ;
\draw    (350,157) .. controls (439.33,157.33) and (441.33,248.67) .. (350,247) ;
\draw    (300,177) .. controls (265.33,182.67) and (268.67,228) .. (300,227) ;
\draw    (20,22) -- (570,22) ;
\draw [shift={(570,22)}, rotate = 180] [color={rgb, 255:red, 0; green, 0; blue, 0 }  ][line width=0.75]    (0,5.59) -- (0,-5.59)   ;
\draw [shift={(20,22)}, rotate = 180] [color={rgb, 255:red, 0; green, 0; blue, 0 }  ][line width=0.75]    (0,5.59) -- (0,-5.59)   ;
\draw    (70,298) -- (630,298) ;
\draw [shift={(630,298)}, rotate = 180] [color={rgb, 255:red, 0; green, 0; blue, 0 }  ][line width=0.75]    (0,5.59) -- (0,-5.59)   ;
\draw [shift={(70,298)}, rotate = 180] [color={rgb, 255:red, 0; green, 0; blue, 0 }  ][line width=0.75]    (0,5.59) -- (0,-5.59)   ;

\draw (22,270.4) node [anchor=north west][inner sep=0.75pt]    {$Y_{11}$};
\draw (62,269.4) node [anchor=north west][inner sep=0.75pt]    {$Y_{10}$};
\draw (551,269.4) node [anchor=north west][inner sep=0.75pt]    {$Y_{01}$};
\draw (601,269.4) node [anchor=north west][inner sep=0.75pt]    {$Y_{00}$};
\draw (251,2.4) node [anchor=north west][inner sep=0.75pt]    {$W_{11,01}$};
\draw (311,300.4) node [anchor=north west][inner sep=0.75pt]    {$W_{10,00}$};
\draw (73,190.4) node [anchor=north west][inner sep=0.75pt]    {$W_{11,10}$};
\draw (571,190.4) node [anchor=north west][inner sep=0.75pt]    {$W_{01,00}$};
\end{tikzpicture}
\caption{Hypersurfaces in a (punctured) cobordism.}
\label{fig:tilde_cob_hypersurfaces}
\end{figure}
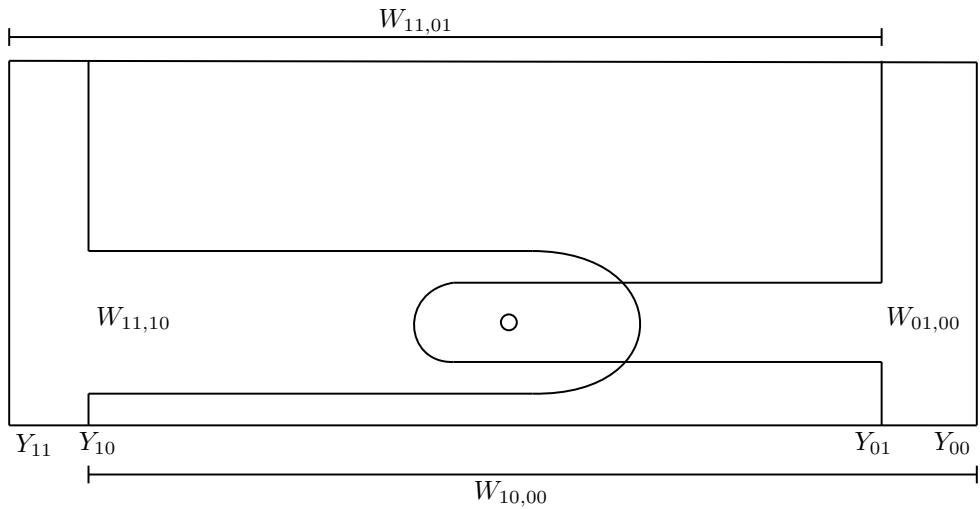

Choose an initial orbifold metric on $(W^{o},\Sigma^{o})$.
We introduce a $1$-parameter family $G_{00,11}$ of orbifold metrics $[-\infty,+\infty]$ by inserting a metric cylinder $[T,0] \times (Y_{10},K_{10})$ if $T < 0$, and a metric cylinder $[0,T] \times (Y_{01},K_{01})$ if $T > 0$.
The metric is broken at $\pm \infty$.
Let $H^{u*}_*(\frakc_{-2} \otimes \ \cdot)$ count the solutions in the zero-dimensional moduli spaces over $G_{01}$:
\[
M(\fraka,\Sigma^{o*},\frakb)
= \bigcup_{\check g \in G_{00,11}} \bigcup_z M_z(\fraka,\frakc_{-2},\Sigma^{o*},\frakb).
\]
Using super- and subscripts to indicated the cobordism for $m$ and $\bar m$, we set
\begin{align*}
	\check H(\upsilon|\Sigma_{11,00}) &=
\begin{bmatrix}
	H^{uo}_o(\frakc_{-2}\otimes \ \cdot) &
	H^{uu}_o(\frakc_{-2} \otimes \delbar^s_u(\cdot)) +\del^u_o \bar H^{us}_u(\frakc_{-2} \otimes \ \cdot)
	\\
	H^{uo}_s(\frakc_{-2}\otimes \ \cdot) &
	H^{us}_s(\frakc_{-2}\otimes \ \cdot)+H^{uu}_s(\frakc_{-2} \otimes \ \delbar^s_u(\cdot)) +\del^u_s \bar H^{us}_u(\frakc_{-2} \otimes \ \cdot)
\end{bmatrix}\\
&\ + 
\begin{bmatrix}
	0 &
		m^u_{o,11,01} \bar m^{us}_{s,01,00}(\frakc_{-2} \otimes \ \cdot)
	+m^u_{o,11,10} \bar m^{us}_{s,01,00}(\frakc_{-2} \otimes \ \cdot) \\
	0 &
		m^{u}_{o,11,01} \bar m^{us}_{s,01,00}(\frakc_{-2} \otimes \ \cdot)
	+m^u_{o,11,10} \bar m^{us}_{s,01,00}(\frakc_{-2} \otimes \ \cdot)
\end{bmatrix}.
\end{align*}
\begin{defn}
The \emph{cobordism map} $\widetilde{\HMR}(\Sigma):\widetilde{\HMR}_{\bullet}(K_1) \to \widetilde{\HMR}_{\bullet}(K_0)$ is defined  as, on the chain level, $\tilde m(\Sigma):\check C(K_1)^{\oplus 2} \to \check C(K_0)^{\oplus 2}$
\[
	\tilde m(\Sigma) = \begin{bmatrix}
		\check m(\Sigma_{10,00}) & 0\\
		\check H(\upsilon \ | \ \Sigma_{11,00}) &
		\check m(\Sigma_{11,01})
	\end{bmatrix}.
\]
\end{defn}
\begin{rem}
The well-definedness of re-interpreted $\upsilon$ and $\tilde m(\Sigma)$ as chain maps can be proved the same way as outlined in \cite[Remark~8.3]{Bloom2011}. In particular, the proof can be extracted from the proof of Lemma~\ref{lem:10_complex} and Lemma~\ref{lem:20_complex}.
\end{rem}
\subsection{Euler characteristics}
\label{subsec:euler_char}
We begin with the Proof of Theorem~\ref{conj:chi=deg} which regarding the equality between $|\chi_{\HMR}|(K,p)$ and Miyazawa's $|\deg(K)|$.
\begin{proof}[Proof of Theorem~\ref{conj:chi=deg}]
Suppose $(K,p)$ is a link with nonzero determinant based at a point $p$.
Since $b^1(\dbcv_2(K))$ is zero, the $\upsilon$ map is independent of base points. 
Once and for all, fix a real spin\textsuperscript{c} structure $(\fraks,\uptau)$ on $\dbcv_2(K)$ lifting the covering involution.
Recall the chain complex $\widetilde C(K)= \widetilde C(K,p,\fraks,\uptau)$ is a sum of two copies of the ``to''-complex
\[
\widecheck C(K) \oplus \widecheck C(K).
\]
We label the generators as follows.
Since $\dbcv_2(K)$ is a rational homolog sphere, we arrange the perturbations so that in the first summand $\widecheck C(K)$, there is a tower $\{\fraka_i:i\in \bbZ_{\ge 0}\}$ of boundary-stable critical points.
Moreover, let $C^o(K)$ be the set of interior (irreducible) critical points.
Fix a $\bbZ$-grading on $\widecheck C(K)$ and hence $\widetilde{C}(K)$ so that \[\gr(\fraka_0,0) = \gr(0,\fraka_0)=0.\]
This is not canonical so the Euler characteristic is defined up to a sign.
See Figure~\ref{fig:237} for a picture of the chain complex for the example $K = P(-2,3,7)$.

Since there are only finitely many irreducibles, there is a sufficiently large $N > \max_{\upgamma}(\gr(\upgamma_{a}))$ such that for all $i \ge N$, we have $\check{\del}(\fraka_i)=0$ and $\upsilon(\fraka_i)=\fraka_{i-1}$.
It follows that for sufficiently high degrees, the differential $\widetilde{\del}$ satisfies
\[
	\widetilde{\del}(\fraka_i,0)=(0,\fraka_{i-1}).
\]
Therefore, the Euler characteristic $\chi(\widetilde{\HMR}_{\bullet}(K,p,\fraks,\uptau))$ can be computed by the alternating sum of dimensions of the $\bbF_2$-vector spaces generated by the following finite set
\[
\big\{(\fraka_i,0): 0 \le i \le N+1\big\} \cup \big\{(0,\fraka_i): 0 \le i \le N\big\} \cup \big\{(\upgamma,0): \upgamma \in C^o(K)\big\}\cup \big\{(0,\upgamma):\upgamma \in C^o(K)\big\}.
\]
The resulting quantity is 
\[
	1 + 2\sum_{\upgamma \in C^o(K)}(-1)^{\gr(\upgamma)}.
\]
But this number is precisely the invariant $|\deg(K)|$ as in \cite[Remark~4.9]{Miyazawa2023}, up to an overall sign.
\end{proof}
Let us compute $|\chi_{\HMR}|$ from the first principle in the following examples.
\begin{exmp}
	Let $\U_{n+1}$ be the $(n+1)$-component unlink based at a point $p$.
	Let $\bbY_n \cong \#_n(S^1 \times S^2)$ be its double branched cover.
	Consider the unique torsion real \spinc structure $(\fraks_{tor},\uptau_{tor})$.
	The Picard torus $H^1(\bbY_n;i\bbR)/H^1(\bbY_n;i\bbZ)$, parametrizing the (real) reducible Seiberg-Witten solutions on $\bbY_n$, is an $n$-torus $\bbT^n$.
	Choose a Morse function $f:\bbT^n \to \reals$ and let $C(\bbT^n)$ be the set of critical points.
	By \cite[Proposition~14.8]{ljk2022}, one can choose the metrics and perturbations so that ``to'' chain complex $\check C = C^s$ is generated by the following set of reducible critical points
	\[\{(\alpha,i): \alpha \in C(\bbT^n), i \in \bbZ_{i \ge 0}.\}\] 
	Moreover, the index-$1$ trajectories contributing to $\check\del$ arise from Flow lines of $f$.
	A choice of a component of $\U_{n+1}$ determines a basis of $H_1(\bbT^n;\bbF_2)$ and
	an endomorphism $\upsilon$.
	Moreover, it determines the following isomorphism:
	\[
	\widecheck{\HMR}_*(\U_{n+1},p)\cong H_*(\bbT^n;\bbF_2) \otimes \bbF[\upsilon,\upsilon^{-1}]/\bbF_2[\upsilon],
	\]
	as an $\bbF_2[\upsilon]$-module.
	To compute $\widetilde{\HMR}_{\bullet}(\U_{n+1},p)$, consider $\check C\oplus \check C$ and
	observe that the mapping cone of $\upsilon$ kills all but the generators of the form $(\alpha_0,0)$, where $\alpha \in C(\bbT^n)$.
	It follows that upon taking homology of $\tilde{\del}$, there is an isomorphism
	\[
		\widetilde{\HMR}_{\bullet}(\U_{n+1},p) \cong H_*(T^n;\bbF_2)\cong \Lambda^*[x_1,\dots,x_n],
	\]
	where the right hand side is the exterior alegbra over $\bbF_2$ of $n$ variables.
	This is an isomorphism of $\bbF_2$-vector spaces, but one can further identify the elements $x_i$ the $\upsilon$-maps at the other $n$ components of $\U_{n+1}$.
	In particular, 
	\[
		\dim \widetilde{\HMR}_{\bullet}(\U_{n+1},p)=2^n.
	\]
	Furthermore, $\chi_{\HMR}(\U_{n+1},p,\fraks,\uptau)$ is nonzero if and only if $n=0$ and $(\fraks,\uptau)=(\fraks_{tor},\uptau_{tor})$.
	In this scenario, we have
	\[
	|\chi_{\HMR}|(\U_1,p,\fraks_{tor},\uptau_{tor}) = 1. 
	\]
	\hfill \qedsymbol
\end{exmp}

We spell out the details of the calculation of $P(-2,3,7)$.
\begin{exmp}
	\label{exmp:237}
	Let $K = P(-2,3,7)$.
	It was known that the double branched cover of $K$ is diffeomorphic to the Brieskorn sphere
	\[
		\bbY = -\Sigma(2,3,7) = \{(z_1, z_2, z_3) \in \bbC^3\big||z_1|^2 + |z_2|^2 + |z_3|^2=1, z_1^2 + z_2^3+z_3^7=0\}.
	\]
	with the orientation reversed. 
	The covering involution can be identified with $\upiota(z_1,z_2,z_3) = (\overline{z_1},\overline{z_2}, \overline{z_3})$.
	There is precisely one real \spinc structure up to equivalence.
	One analyzes the Seiberg-Witten solutions on $\bbY$ using the description in \cite{MOY}. 
	For a choice of $\upiota$-invariant metric and regular perturbation, there are precisely two irreducible solutions and one reducible solution.
	In addition, these Seiberg-Witten solutions are invariant under the real structure and generate the $\HMR$ chain complexes.
	
	Therefore, the $\HMR$-``to'' complex $\check{C}$ is generated by two irreducible critical points $\alpha$ and $\beta$, and a tower $\{\fraka_i: i \in \bbZ_{\ge 0}\}$ of boundary-stable reducible critical points.
	In particular, $\fraka_0$ is one degree higher than both $\alpha$ and $\beta$.
	The differential $\check \del$ is zero, while $\upsilon(\fraka_i) = \fraka_{i-1}$ whenever $i > 0$, and $\upsilon$ sends $\fraka_0$ to $\alpha+\beta$.
	The latter fact follows from \cite{MOY}'s description of the moduli space from an irreducible to the reducible, where the real moduli space $M^+(\fraka_0,\alpha) \cong \mathbb{RP}^1$ on which $\upsilon$ evaluates to $1$.
	Same holds for $M^+(\fraka_0,\beta)$.
	See Figure~\ref{fig:237} for the action of $\check{\del}$.
	To describe the ``tilde'' complex, we take two copies of $\check C(K)$, where the first copy has generators
\[(\alpha,0), (\beta,0), \text{ and } \{(\fraka_i,0): i \in \bbZ_{\ge 0}\}.\]
The differential is given by
\begin{align*}
\tilde{\del}(\fraka_0,0) 
&= (0,\upsilon(\fraka_0))= (0,\alpha+\beta),\\
\tilde{\del}(\fraka_i,0)
&=
(0,\upsilon(\fraka_i))= (0,\fraka_{i-1}),
\end{align*}
where $i > 1$.
We record the lower triangular entry $\upsilon$ of $\tilde{\del}$ in the right hand side of Figure~\ref{fig:237}.
We see that $\ker\tilde{\del}$ is generated by the four irreducibles and $\frakb_i$ and, dividing by $\text{im} \ \tilde{\del}$ (which in particular contains $(0,\alpha+\beta)$), the resulting group is generated by $3$ irreducibles all supported on the same degree, i.e.
\[
\left|\chi(\widetilde{\HMR}_{\bullet}(K))\right| = \dim \widetilde{\HMR}_{\bullet}(K)=  3.
\]

\begin{figure}[ht]
\centering
\tikzset{every picture/.style={line width=0.75pt}} 
\begin{tikzpicture}[x=0.57pt,y=0.57pt,yscale=-1,xscale=1]
	
	\draw  [fill={rgb, 255:red, 0; green, 0; blue, 0 }  ,fill opacity=1 ] (320,285) .. controls (320,282.24) and (322.24,280) .. (325,280) .. controls (327.76,280) and (330,282.24) .. (330,285) .. controls (330,287.76) and (327.76,290) .. (325,290) .. controls (322.24,290) and (320,287.76) .. (320,285) -- cycle ;
	\draw  [fill={rgb, 255:red, 0; green, 0; blue, 0 }  ,fill opacity=1 ] (410,285) .. controls (410,282.24) and (412.24,280) .. (415,280) .. controls (417.76,280) and (420,282.24) .. (420,285) .. controls (420,287.76) and (417.76,290) .. (415,290) .. controls (412.24,290) and (410,287.76) .. (410,285) -- cycle ;
	\draw  [fill={rgb, 255:red, 0; green, 0; blue, 0 }  ,fill opacity=1 ] (370,235) .. controls (370,232.24) and (372.24,230) .. (375,230) .. controls (377.76,230) and (380,232.24) .. (380,235) .. controls (380,237.76) and (377.76,240) .. (375,240) .. controls (372.24,240) and (370,237.76) .. (370,235) -- cycle ;
	\draw  [fill={rgb, 255:red, 0; green, 0; blue, 0 }  ,fill opacity=1 ] (550,285) .. controls (550,282.24) and (552.24,280) .. (555,280) .. controls (557.76,280) and (560,282.24) .. (560,285) .. controls (560,287.76) and (557.76,290) .. (555,290) .. controls (552.24,290) and (550,287.76) .. (550,285) -- cycle ;
	\draw  [fill={rgb, 255:red, 0; green, 0; blue, 0 }  ,fill opacity=1 ] (660,285) .. controls (660,282.24) and (662.24,280) .. (665,280) .. controls (667.76,280) and (670,282.24) .. (670,285) .. controls (670,287.76) and (667.76,290) .. (665,290) .. controls (662.24,290) and (660,287.76) .. (660,285) -- cycle ;
	\draw  [fill={rgb, 255:red, 0; green, 0; blue, 0 }  ,fill opacity=1 ] (600,235) .. controls (600,232.24) and (602.24,230) .. (605,230) .. controls (607.76,230) and (610,232.24) .. (610,235) .. controls (610,237.76) and (607.76,240) .. (605,240) .. controls (602.24,240) and (600,237.76) .. (600,235) -- cycle ;
	\draw  [fill={rgb, 255:red, 0; green, 0; blue, 0 }  ,fill opacity=1 ] (370,185) .. controls (370,182.24) and (372.24,180) .. (375,180) .. controls (377.76,180) and (380,182.24) .. (380,185) .. controls (380,187.76) and (377.76,190) .. (375,190) .. controls (372.24,190) and (370,187.76) .. (370,185) -- cycle ;
	\draw  [fill={rgb, 255:red, 0; green, 0; blue, 0 }  ,fill opacity=1 ] (370,135) .. controls (370,132.24) and (372.24,130) .. (375,130) .. controls (377.76,130) and (380,132.24) .. (380,135) .. controls (380,137.76) and (377.76,140) .. (375,140) .. controls (372.24,140) and (370,137.76) .. (370,135) -- cycle ;
	\draw  [fill={rgb, 255:red, 0; green, 0; blue, 0 }  ,fill opacity=1 ] (370,85) .. controls (370,82.24) and (372.24,80) .. (375,80) .. controls (377.76,80) and (380,82.24) .. (380,85) .. controls (380,87.76) and (377.76,90) .. (375,90) .. controls (372.24,90) and (370,87.76) .. (370,85) -- cycle ;
	\draw  [fill={rgb, 255:red, 0; green, 0; blue, 0 }  ,fill opacity=1 ] (600,135) .. controls (600,132.24) and (602.24,130) .. (605,130) .. controls (607.76,130) and (610,132.24) .. (610,135) .. controls (610,137.76) and (607.76,140) .. (605,140) .. controls (602.24,140) and (600,137.76) .. (600,135) -- cycle ;
	\draw  [fill={rgb, 255:red, 0; green, 0; blue, 0 }  ,fill opacity=1 ] (600,85) .. controls (600,82.24) and (602.24,80) .. (605,80) .. controls (607.76,80) and (610,82.24) .. (610,85) .. controls (610,87.76) and (607.76,90) .. (605,90) .. controls (602.24,90) and (600,87.76) .. (600,85) -- cycle ;
	\draw  [fill={rgb, 255:red, 0; green, 0; blue, 0 }  ,fill opacity=1 ] (600,185) .. controls (600,182.24) and (602.24,180) .. (605,180) .. controls (607.76,180) and (610,182.24) .. (610,185) .. controls (610,187.76) and (607.76,190) .. (605,190) .. controls (602.24,190) and (600,187.76) .. (600,185) -- cycle ;
	\draw  [fill={rgb, 255:red, 0; green, 0; blue, 0 }  ,fill opacity=1 ] (40,285) .. controls (40,282.24) and (42.24,280) .. (45,280) .. controls (47.76,280) and (50,282.24) .. (50,285) .. controls (50,287.76) and (47.76,290) .. (45,290) .. controls (42.24,290) and (40,287.76) .. (40,285) -- cycle ;
	\draw  [fill={rgb, 255:red, 0; green, 0; blue, 0 }  ,fill opacity=1 ] (140,285) .. controls (140,282.24) and (142.24,280) .. (145,280) .. controls (147.76,280) and (150,282.24) .. (150,285) .. controls (150,287.76) and (147.76,290) .. (145,290) .. controls (142.24,290) and (140,287.76) .. (140,285) -- cycle ;
	\draw  [fill={rgb, 255:red, 0; green, 0; blue, 0 }  ,fill opacity=1 ] (90,234.84) .. controls (90.09,232.08) and (92.4,229.91) .. (95.16,230) .. controls (97.92,230.09) and (100.09,232.4) .. (100,235.16) .. controls (99.91,237.92) and (97.6,240.09) .. (94.84,240) .. controls (92.08,239.91) and (89.91,237.6) .. (90,234.84) -- cycle ;
	\draw  [fill={rgb, 255:red, 0; green, 0; blue, 0 }  ,fill opacity=1 ] (90,185) .. controls (90,182.24) and (92.24,180) .. (95,180) .. controls (97.76,180) and (100,182.24) .. (100,185) .. controls (100,187.76) and (97.76,190) .. (95,190) .. controls (92.24,190) and (90,187.76) .. (90,185) -- cycle ;
	\draw  [fill={rgb, 255:red, 0; green, 0; blue, 0 }  ,fill opacity=1 ] (90,135) .. controls (90,132.24) and (92.24,130) .. (95,130) .. controls (97.76,130) and (100,132.24) .. (100,135) .. controls (100,137.76) and (97.76,140) .. (95,140) .. controls (92.24,140) and (90,137.76) .. (90,135) -- cycle ;
	\draw  [fill={rgb, 255:red, 0; green, 0; blue, 0 }  ,fill opacity=1 ] (90,85) .. controls (90,82.24) and (92.24,80) .. (95,80) .. controls (97.76,80) and (100,82.24) .. (100,85) .. controls (100,87.76) and (97.76,90) .. (95,90) .. controls (92.24,90) and (90,87.76) .. (90,85) -- cycle ;
	\draw    (230,0) -- (230,310) ;
	\draw    (90,240) -- (52.12,277.88) ;
	\draw [shift={(50,280)}, rotate = 315] [fill={rgb, 255:red, 0; green, 0; blue, 0 }  ][line width=0.08]  [draw opacity=0] (10.72,-5.15) -- (0,0) -- (10.72,5.15) -- (7.12,0) -- cycle    ;
	\draw    (95,185) -- (95.15,227) ;
	\draw [shift={(95.16,230)}, rotate = 269.8] [fill={rgb, 255:red, 0; green, 0; blue, 0 }  ][line width=0.08]  [draw opacity=0] (10.72,-5.15) -- (0,0) -- (10.72,5.15) -- (7.12,0) -- cycle    ;
	\draw    (95,140) -- (95,177) ;
	\draw [shift={(95,180)}, rotate = 270] [fill={rgb, 255:red, 0; green, 0; blue, 0 }  ][line width=0.08]  [draw opacity=0] (10.72,-5.15) -- (0,0) -- (10.72,5.15) -- (7.12,0) -- cycle    ;
	\draw    (95,90) -- (95,127) ;
	\draw [shift={(95,130)}, rotate = 270] [fill={rgb, 255:red, 0; green, 0; blue, 0 }  ][line width=0.08]  [draw opacity=0] (10.72,-5.15) -- (0,0) -- (10.72,5.15) -- (7.12,0) -- cycle    ;
	\draw    (95,40) -- (95,77) ;
	\draw [shift={(95,80)}, rotate = 270] [fill={rgb, 255:red, 0; green, 0; blue, 0 }  ][line width=0.08]  [draw opacity=0] (10.72,-5.15) -- (0,0) -- (10.72,5.15) -- (7.12,0) -- cycle    ;
	\draw [color={rgb, 255:red, 208; green, 2; blue, 27 }  ,draw opacity=1 ]   (380,35) -- (597.07,84.34) ;
	\draw [shift={(600,85)}, rotate = 192.8] [fill={rgb, 255:red, 208; green, 2; blue, 27 }  ,fill opacity=1 ][line width=0.08]  [draw opacity=0] (10.72,-5.15) -- (0,0) -- (10.72,5.15) -- (7.12,0) -- cycle    ;
	\draw [color={rgb, 255:red, 208; green, 2; blue, 27 }  ,draw opacity=1 ]   (380,85) -- (597.07,134.34) ;
	\draw [shift={(600,135)}, rotate = 192.8] [fill={rgb, 255:red, 208; green, 2; blue, 27 }  ,fill opacity=1 ][line width=0.08]  [draw opacity=0] (10.72,-5.15) -- (0,0) -- (10.72,5.15) -- (7.12,0) -- cycle    ;
	\draw [color={rgb, 255:red, 208; green, 2; blue, 27 }  ,draw opacity=1 ]   (380,135) -- (597.07,184.34) ;
	\draw [shift={(600,185)}, rotate = 192.8] [fill={rgb, 255:red, 208; green, 2; blue, 27 }  ,fill opacity=1 ][line width=0.08]  [draw opacity=0] (10.72,-5.15) -- (0,0) -- (10.72,5.15) -- (7.12,0) -- cycle    ;
	\draw [color={rgb, 255:red, 208; green, 2; blue, 27 }  ,draw opacity=1 ]   (380,235) .. controls (459.68,253.24) and (629.37,279.16) .. (657.27,284.45) ;
	\draw [shift={(660,285)}, rotate = 192.65] [fill={rgb, 255:red, 208; green, 2; blue, 27 }  ,fill opacity=1 ][line width=0.08]  [draw opacity=0] (10.72,-5.15) -- (0,0) -- (10.72,5.15) -- (7.12,0) -- cycle    ;
	\draw    (100,240) -- (137.88,277.88) ;
	\draw [shift={(140,280)}, rotate = 225] [fill={rgb, 255:red, 0; green, 0; blue, 0 }  ][line width=0.08]  [draw opacity=0] (10.72,-5.15) -- (0,0) -- (10.72,5.15) -- (7.12,0) -- cycle    ;
	\draw [color={rgb, 255:red, 208; green, 2; blue, 27 }  ,draw opacity=1 ]   (380,185) -- (597.07,234.34) ;
	\draw [shift={(600,235)}, rotate = 192.8] [fill={rgb, 255:red, 208; green, 2; blue, 27 }  ,fill opacity=1 ][line width=0.08]  [draw opacity=0] (10.72,-5.15) -- (0,0) -- (10.72,5.15) -- (7.12,0) -- cycle    ;
	\draw [color={rgb, 255:red, 208; green, 2; blue, 27 }  ,draw opacity=1 ]   (380,235) -- (547.12,284.15) ;
	\draw [shift={(550,285)}, rotate = 196.39] [fill={rgb, 255:red, 208; green, 2; blue, 27 }  ,fill opacity=1 ][line width=0.08]  [draw opacity=0] (10.72,-5.15) -- (0,0) -- (10.72,5.15) -- (7.12,0) -- cycle    ;
	
	\draw (324,62.4) node [anchor=north west][inner sep=0.75pt]    {$( \fraka _{3} ,0)$};
	\draw (363,22) node [anchor=north west][inner sep=0.75pt]   [align=left] {$\displaystyle \vdots $};
	\draw (280,262.4) node [anchor=north west][inner sep=0.75pt]    {$( \alpha ,0)$};
	\draw (557,283.4) node [anchor=north west][inner sep=0.75pt]    {$( 0,\alpha )$};
	\draw (371,262.4) node [anchor=north west][inner sep=0.75pt]    {$( \beta ,0)$};
	\draw (667,283.4) node [anchor=north west][inner sep=0.75pt]    {$( \beta ,0)$};
	\draw (101,212.4) node [anchor=north west][inner sep=0.75pt]    {$\fraka _{0}$};
	\draw (101,172.4) node [anchor=north west][inner sep=0.75pt]    {$\fraka _{1}$};
	\draw (101,122.4) node [anchor=north west][inner sep=0.75pt]    {$\fraka _{2}$};
	\draw (101,72.4) node [anchor=north west][inner sep=0.75pt]    {$\fraka _{3}$};
	\draw (91,20) node [anchor=north west][inner sep=0.75pt]   [align=left] {$\displaystyle \vdots $};
	\draw (21,260.4) node [anchor=north west][inner sep=0.75pt]    {$\alpha _{\fraka }$};
	\draw (151,260.4) node [anchor=north west][inner sep=0.75pt]    {$\beta _{\fraka }$};
	\draw (593,22) node [anchor=north west][inner sep=0.75pt]   [align=left] {$\displaystyle \vdots $};
	\draw (324,110.4) node [anchor=north west][inner sep=0.75pt]    {$( \fraka _{2} ,0)$};
	\draw (324,160.4) node [anchor=north west][inner sep=0.75pt]    {$( \fraka _{1} ,0)$};
	\draw (324,212.4) node [anchor=north west][inner sep=0.75pt]    {$( \fraka _{0} ,0)$};
	\draw (611,82.4) node [anchor=north west][inner sep=0.75pt]    {$( 0,\fraka _{3})$};
	\draw (611,130.4) node [anchor=north west][inner sep=0.75pt]    {$( 0,\fraka _{2})$};
	\draw (611,182.4) node [anchor=north west][inner sep=0.75pt]    {$( 0,\fraka _{1})$};
	\draw (611,230.4) node [anchor=north west][inner sep=0.75pt]    {$( 0,\fraka _{0})$};
\end{tikzpicture}
\caption{Action of $\upsilon$ (left) and the $\widetilde{\HMR}$-differential (right).}
\label{fig:237}
\end{figure}
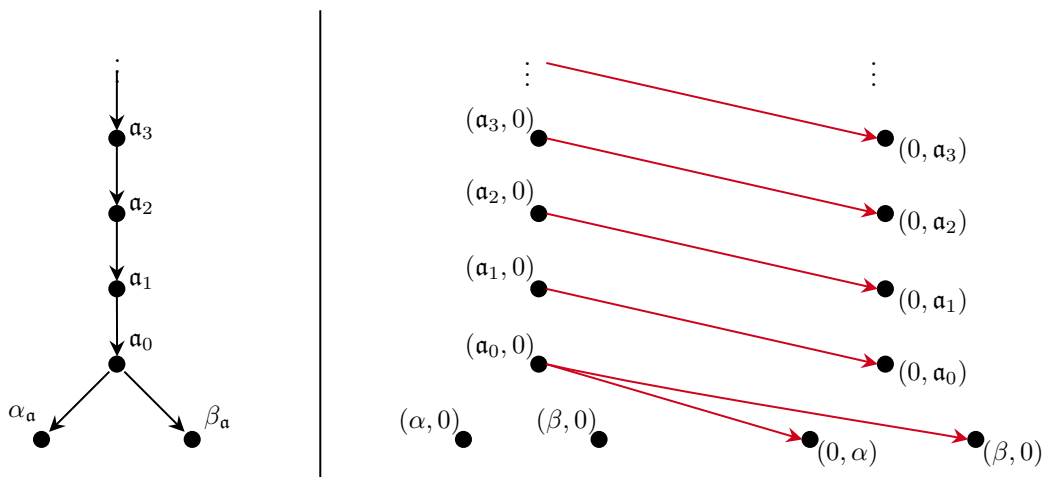
\end{exmp}
\begin{rem}
The monopole Floer homology of $-\Sigma(2,3,7)$ admits a similar description.
The key difference is that the $U$ mapping cone 	is
\[\widetilde{\textit{C}}_{\HM}=\widecheck{\textit{C}}_{\HM} \oplus \widecheck{\textit{C}}_{\HM}\langle1\rangle,\]
where $\la 1 \ra$ indicate a degree shift by one, as now $\deg U=-2$.
Since in the ordinary case $\{(\alpha,0),(\beta,0)\}$ and $\{(0,\alpha),(0,\beta)\}$ have different mod-two degrees, the resulting Euler characteristic is $1$.
\end{rem}

\begin{exmp}
	Let $K$ be the torus knot $T(p,q)$  where $p,q$ are odd.
	As observed in \cite[Section~14.5]{ljk2022}, $\check C(K)$ can be chosen to consist of only reducibles and $\widetilde{\HMR}_{\bullet}(K) \cong \widetilde{\HMR}_{\bullet}(\U_1)$, where we choose the unique real \spinc structure on $\dbcv_2(K)=\Sigma(2,p,q)$ and the unique torsion real \spinc structure on $\dbcv_2(\U_1)\cong S^3$.
	It follows that $\widetilde{\HMR}_{\bullet}(K) \cong \bbF_2$, and $|\chi(\widetilde{\HMR}_{\bullet}(K))|=1$.
	On the other hand, $\widetilde{\HM}_{\bullet}(\Sigma(2,p,q))$ has irreducible generators and higher rank. Still, $|\chi(\widetilde{\HM}_{\bullet}(\Sigma(2,p,q)))|=1$.
\end{exmp}

\begin{exmp}
	Let $K$ be a $2$-bridge link.
	Then $\dbcv_2(K)$ is diffeomorphic to a lens space and there  exists an involution invariant Riemmannien metric on $\dbcv_2(K)$ having positive scalar curvature (see \cite[Section~14.2]{ljk2022}).
	In particular, there exists no irreducible Seiberg-Witten solutions, and by the arguments above, for each real \spinc structure $(\fraks,\uptau)$:
	\[\dim_{\bbF_2}\widetilde{\HMR}_{\bullet}(K,\fraks, \uptau) = |\chi_{\HMR}|(K,\fraks,\uptau) = 1.\]
	Moreover, the total rank is equal to the number of \spinc structures, which is equal to the determinant of $K$.
\end{exmp}

\begin{rem}
\label{rem:noskein}
Let $K_2 = P(-2,3,7)$ and consider the crossing in the circle of Figure~\ref{fig:skein237}. 
The $1$-resolution $K_1$ is equal to the $7_2$ knot, which is a twist knot with $5$ half-twists; 
the $0$-resolution $K_0$ is the two-component torus link $L10a18$.
Both $K_1$ and $K_0$ are $2$-bridge with determinants equal to $|\det K_1|= 11$ and $|\det K_0|=10$, respectively.
It follows that
\[
\widetilde{\HMR}_{\bullet}(K_2) \cong \bbF_2^{\oplus 3}, \quad
\widetilde{\HMR}_{\bullet}(K_1) \cong \bigoplus_{\fraks_i; 1 \le i \le 11} \bbF_2, \quad
\widetilde{\HMR}_{\bullet}(K_0)\cong 
\bigoplus_{\frakt_j; 1 \le j \le 10}\bbF_2,
\]
where $\fraks_i$ is a \spinc structure on $\dbcv_2(K_1)$, for $1 \le i \le 11$, and $\frakt_j$ is a \spinc structure on $\dbcv_2(K_0)$ for $1 \le j \le 10$.
\end{rem}

\begin{figure}[ht]
\centering

\tikzset{every picture/.style={line width=0.75pt}} 

\tikzset{every picture/.style={line width=0.75pt}} 

\begin{tikzpicture}[x=0.55pt,y=0.55pt,yscale=-1,xscale=1]
	
	\draw    (25.18,39.41) .. controls (63.53,10.27) and (154.84,14.64) .. (192.47,42.33) ;
	\draw    (24.46,132.17) .. controls (25.18,147.23) and (43.4,164.23) .. (47.71,175.4) ;
	\draw    (52.26,132.66) .. controls (52.02,146.26) and (44.35,151.6) .. (38.6,156.94) ;
	\draw    (31.89,161.31) .. controls (28.3,166.17) and (24.7,166.66) .. (23.74,171.03) ;
	\draw    (23.74,171.03) .. controls (18.47,183.17) and (48.67,210.37) .. (51.07,225.42) ;
	\draw    (47.71,175.4) .. controls (51.07,184.14) and (42.92,191.43) .. (37.16,196.77) ;
	\draw    (33.09,200.17) .. controls (29.5,205.03) and (24.7,211.34) .. (24.7,226.88) ;
	\draw    (122.49,116.63) .. controls (122.25,130.23) and (101.4,148.69) .. (95.64,154.03) ;
	\draw    (95.16,115.66) .. controls (95.64,124.89) and (106.19,132.17) .. (108.59,135.57) ;
	\draw    (114.34,140.92) .. controls (115.78,147.23) and (120.09,145.77) .. (123.44,153.54) ;
	\draw    (123.44,153.54) .. controls (123.2,167.14) and (105.23,187.06) .. (99.48,192.4) ;
	\draw    (99.48,192.4) .. controls (99.96,201.63) and (110.5,209.4) .. (112.9,212.8) ;
	
	\draw    (96.12,153.54) .. controls (96.6,162.77) and (107.15,170.06) .. (109.54,173.46) ;
	\draw    (115.3,178.8) .. controls (116.73,185.11) and (121.05,183.66) .. (124.4,191.43) ;
	\draw    (124.4,191.43) .. controls (124.16,205.03) and (102.83,225.42) .. (97.08,230.77) ;
	\draw    (115.78,217.65) .. controls (118.17,222.51) and (121.53,222.51) .. (124.88,230.28) ;
	\draw    (191.51,61.27) .. controls (191.27,74.87) and (170.42,93.32) .. (164.67,98.66) ;
	\draw    (158.44,48.64) .. controls (165.15,55.44) and (175.21,76.81) .. (177.61,80.21) ;
	\draw    (183.36,85.55) .. controls (184.8,91.86) and (189.11,90.41) .. (192.47,98.18) ;
	\draw    (192.47,98.18) .. controls (192.23,111.78) and (171.38,135.09) .. (165.63,140.43) ;
	\draw    (164.67,98.66) .. controls (165.15,107.89) and (176.17,118.58) .. (178.57,121.98) ;
	\draw    (184.32,127.32) .. controls (185.76,133.63) and (190.07,132.17) .. (193.43,139.95) ;
	\draw    (193.43,139.95) .. controls (193.19,153.54) and (174.73,168.11) .. (170.9,181.71) ;
	\draw    (165.63,140.43) .. controls (166.1,149.66) and (176.65,158.89) .. (179.05,162.29) ;
	\draw    (184.8,167.63) .. controls (187.67,170.06) and (190.55,172.49) .. (193.91,180.26) ;
	\draw    (193.91,180.26) .. controls (193.67,193.85) and (180.96,206.97) .. (168.98,219.11) ;
	\draw    (170.9,181.71) .. controls (171.38,190.94) and (179.53,197.25) .. (181.92,200.65) ;
	\draw    (187.67,206) .. controls (191.99,208.42) and (193.43,210.85) .. (196.78,218.62) ;
	\draw    (196.78,218.62) .. controls (196.54,232.22) and (175.21,251.65) .. (169.46,256.99) ;
	\draw    (168.98,219.11) .. controls (163.23,228.34) and (180.01,235.14) .. (182.4,238.54) ;
	\draw    (188.15,243.88) .. controls (189.59,250.19) and (193.91,248.74) .. (197.26,256.51) ;
	\draw    (197.26,256.51) .. controls (197.02,270.11) and (174.73,290.5) .. (168.98,295.85) ;
	\draw    (169.46,256.99) .. controls (169.94,266.22) and (179.53,273.99) .. (181.92,277.39) ;
	\draw    (187.67,282.73) .. controls (189.11,289.05) and (193.43,287.59) .. (196.78,295.36) ;
	\draw    (196.78,295.36) .. controls (196.54,308.96) and (175.69,328.39) .. (169.94,333.73) ;
	\draw    (169.46,295.36) .. controls (169.94,304.59) and (180.48,311.87) .. (182.88,315.27) ;
	\draw    (187.67,319.16) .. controls (189.11,325.47) and (193.43,324.01) .. (196.78,331.79) ;
	\draw    (192.47,42.33) .. controls (193.43,44.75) and (194.86,52.04) .. (191.51,61.27) ;
	\draw    (48.43,51.07) .. controls (63.53,36.01) and (79.11,35.53) .. (95.16,50.1) ;
	\draw    (111.7,49.61) .. controls (126.8,34.55) and (142.38,34.07) .. (158.44,48.64) ;
	\draw    (25.18,39.41) .. controls (12.24,45.24) and (8.17,87.49) .. (24.46,132.17) ;
	\draw    (48.43,51.07) .. controls (34.77,87.01) and (57.78,93.32) .. (52.26,132.66) ;
	\draw    (95.16,50.1) .. controls (102.35,76.32) and (88.93,84.09) .. (95.16,115.66) ;
	\draw    (111.7,49.61) .. controls (101.4,79.24) and (125.36,99.63) .. (122.49,116.63) ;
	\draw    (29.5,345.87) .. controls (60.17,391.52) and (177.13,373.55) .. (196.78,348.78) ;
	\draw    (196.78,331.79) .. controls (205.89,340.04) and (197.26,349.76) .. (196.78,348.78) ;
	\draw    (124.88,230.28) .. controls (141.66,258.45) and (105.23,258.93) .. (123.2,334.7) ;
	\draw    (55.14,333.73) .. controls (68.8,355.1) and (94.68,352.18) .. (101.87,332.76) ;
	\draw    (123.2,334.7) .. controls (136.87,356.07) and (162.75,353.15) .. (169.94,333.73) ;
	\draw    (97.08,230.77) .. controls (80.3,261.36) and (110.02,252.62) .. (101.87,332.76) ;
	\draw    (24.7,226.88) .. controls (42.44,275.45) and (1.22,257.96) .. (29.5,345.87) ;
	\draw    (51.07,225.42) .. controls (68.8,273.99) and (26.86,245.82) .. (55.14,333.73) ;
	\draw  [color={rgb, 255:red, 230; green, 22; blue, 22 }  ,draw opacity=1 ] (12,153.79) .. controls (12,139.71) and (23.27,128.29) .. (37.16,128.29) .. controls (51.06,128.29) and (62.33,139.71) .. (62.33,153.79) .. controls (62.33,167.87) and (51.06,179.28) .. (37.16,179.28) .. controls (23.27,179.28) and (12,167.87) .. (12,153.79) -- cycle ;
	\draw    (250.09,40.88) .. controls (288.43,11.74) and (379.75,16.11) .. (417.37,43.8) ;
	\draw    (249.37,133.65) .. controls (257.75,138.5) and (264.75,155.5) .. (256.8,162.79) ;
	\draw    (277.17,134.13) .. controls (264.75,145) and (265.75,155) .. (272.62,176.87) ;
	\draw    (256.8,162.79) .. controls (253.2,167.64) and (249.61,168.13) .. (248.65,172.5) ;
	\draw    (248.65,172.5) .. controls (243.38,184.64) and (273.57,211.84) .. (275.97,226.89) ;
	\draw    (272.62,176.87) .. controls (275.97,185.61) and (267.82,192.9) .. (262.07,198.24) ;
	\draw    (258,201.64) .. controls (254.4,206.5) and (249.61,212.81) .. (249.61,228.35) ;
	\draw    (347.39,118.1) .. controls (347.15,131.7) and (326.3,150.16) .. (320.55,155.5) ;
	\draw    (320.07,117.13) .. controls (320.55,126.36) and (331.09,133.65) .. (333.49,137.05) ;
	\draw    (339.24,142.39) .. controls (340.68,148.7) and (344.99,147.24) .. (348.35,155.02) ;
	\draw    (348.35,155.02) .. controls (348.11,168.61) and (330.14,188.53) .. (324.38,193.87) ;
	\draw    (324.38,193.87) .. controls (324.86,203.1) and (335.41,210.87) .. (337.8,214.27) ;
	
	\draw    (321.03,155.02) .. controls (321.51,164.24) and (332.05,171.53) .. (334.45,174.93) ;
	\draw    (340.2,180.27) .. controls (341.64,186.58) and (345.95,185.13) .. (349.31,192.9) ;
	\draw    (349.31,192.9) .. controls (349.07,206.5) and (327.74,226.89) .. (321.99,232.24) ;
	\draw    (340.68,219.12) .. controls (343.08,223.98) and (346.43,223.98) .. (349.79,231.75) ;
	\draw    (416.42,62.74) .. controls (416.18,76.34) and (395.32,94.79) .. (389.57,100.13) ;
	\draw    (383.34,50.11) .. controls (390.05,56.91) and (400.12,78.28) .. (402.51,81.68) ;
	\draw    (408.27,87.02) .. controls (409.7,93.34) and (414.02,91.88) .. (417.37,99.65) ;
	\draw    (417.37,99.65) .. controls (417.13,113.25) and (396.28,136.56) .. (390.53,141.9) ;
	\draw    (389.57,100.13) .. controls (390.05,109.36) and (401.08,120.05) .. (403.47,123.45) ;
	\draw    (409.23,128.79) .. controls (410.66,135.1) and (414.98,133.65) .. (418.33,141.42) ;
	\draw    (418.33,141.42) .. controls (418.09,155.02) and (399.64,169.59) .. (395.8,183.18) ;
	\draw    (390.53,141.9) .. controls (391.01,151.13) and (401.56,160.36) .. (403.95,163.76) ;
	\draw    (409.7,169.1) .. controls (412.58,171.53) and (415.46,173.96) .. (418.81,181.73) ;
	\draw    (418.81,181.73) .. controls (418.57,195.33) and (405.87,208.44) .. (393.89,220.58) ;
	\draw    (395.8,183.18) .. controls (396.28,192.41) and (404.43,198.73) .. (406.83,202.13) ;
	\draw    (412.58,207.47) .. controls (416.89,209.9) and (418.33,212.32) .. (421.69,220.1) ;
	\draw    (421.69,220.1) .. controls (421.45,233.69) and (400.12,253.12) .. (394.37,258.46) ;
	\draw    (393.89,220.58) .. controls (388.13,229.81) and (404.91,236.61) .. (407.31,240.01) ;
	\draw    (413.06,245.35) .. controls (414.5,251.66) and (418.81,250.21) .. (422.17,257.98) ;
	\draw    (422.17,257.98) .. controls (421.93,271.58) and (399.64,291.97) .. (393.89,297.32) ;
	\draw    (394.37,258.46) .. controls (394.85,267.69) and (404.43,275.46) .. (406.83,278.86) ;
	\draw    (412.58,284.2) .. controls (414.02,290.52) and (418.33,289.06) .. (421.69,296.83) ;
	\draw    (421.69,296.83) .. controls (421.45,310.43) and (400.6,329.86) .. (394.85,335.2) ;
	\draw    (394.37,296.83) .. controls (394.85,306.06) and (405.39,313.34) .. (407.79,316.74) ;
	\draw    (412.58,320.63) .. controls (414.02,326.94) and (418.33,325.49) .. (421.69,333.26) ;
	\draw    (417.37,43.8) .. controls (418.33,46.23) and (419.77,53.51) .. (416.42,62.74) ;
	\draw    (273.33,52.54) .. controls (288.43,37.48) and (304.01,37) .. (320.07,51.57) ;
	\draw    (336.61,51.08) .. controls (351.71,36.03) and (367.28,35.54) .. (383.34,50.11) ;
	\draw    (250.09,40.88) .. controls (237.15,46.71) and (233.07,88.96) .. (249.37,133.65) ;
	\draw    (273.33,52.54) .. controls (259.67,88.48) and (282.68,94.79) .. (277.17,134.13) ;
	\draw    (320.07,51.57) .. controls (327.26,77.79) and (313.84,85.56) .. (320.07,117.13) ;
	\draw    (336.61,51.08) .. controls (326.3,80.71) and (350.27,101.11) .. (347.39,118.1) ;
	\draw    (254.4,347.34) .. controls (285.08,392.99) and (402.04,375.02) .. (421.69,350.26) ;
	\draw    (421.69,333.26) .. controls (430.8,341.51) and (422.17,351.23) .. (421.69,350.26) ;
	\draw    (349.79,231.75) .. controls (366.56,259.92) and (330.14,260.41) .. (348.11,336.17) ;
	\draw    (280.05,335.2) .. controls (293.71,356.57) and (319.59,353.65) .. (326.78,334.23) ;
	\draw    (348.11,336.17) .. controls (361.77,357.54) and (387.66,354.63) .. (394.85,335.2) ;
	\draw    (321.99,232.24) .. controls (305.21,262.83) and (334.93,254.09) .. (326.78,334.23) ;
	\draw    (249.61,228.35) .. controls (267.34,276.92) and (226.12,259.43) .. (254.4,347.34) ;
	\draw    (275.97,226.89) .. controls (293.71,275.46) and (251.77,247.29) .. (280.05,335.2) ;
	\draw    (479.49,40.88) .. controls (517.84,11.74) and (609.15,16.11) .. (646.78,43.8) ;
	\draw    (506.58,134.13) .. controls (506.34,147.73) and (487.25,151) .. (478.77,133.65) ;
	\draw    (502.02,176.87) .. controls (495.75,152.5) and (479.01,168.13) .. (478.06,172.5) ;
	\draw    (478.06,172.5) .. controls (472.78,184.64) and (502.98,211.84) .. (505.38,226.89) ;
	\draw    (502.02,176.87) .. controls (505.38,185.61) and (497.23,192.9) .. (491.48,198.24) ;
	\draw    (487.4,201.64) .. controls (483.81,206.5) and (479.01,212.81) .. (479.01,228.35) ;
	\draw    (576.8,118.1) .. controls (576.56,131.7) and (555.71,150.16) .. (549.95,155.5) ;
	\draw    (549.48,117.13) .. controls (549.95,126.36) and (560.5,133.65) .. (562.9,137.05) ;
	\draw    (568.65,142.39) .. controls (570.09,148.7) and (574.4,147.24) .. (577.76,155.02) ;
	\draw    (577.76,155.02) .. controls (577.52,168.61) and (559.54,188.53) .. (553.79,193.87) ;
	\draw    (553.79,193.87) .. controls (554.27,203.1) and (564.81,210.87) .. (567.21,214.27) ;
	
	\draw    (550.43,155.02) .. controls (550.91,164.24) and (561.46,171.53) .. (563.86,174.93) ;
	\draw    (569.61,180.27) .. controls (571.05,186.58) and (575.36,185.13) .. (578.71,192.9) ;
	\draw    (578.71,192.9) .. controls (578.47,206.5) and (557.14,226.89) .. (551.39,232.24) ;
	\draw    (570.09,219.12) .. controls (572.48,223.98) and (575.84,223.98) .. (579.19,231.75) ;
	\draw    (645.82,62.74) .. controls (645.58,76.34) and (624.73,94.79) .. (618.98,100.13) ;
	\draw    (612.75,50.11) .. controls (619.46,56.91) and (629.52,78.28) .. (631.92,81.68) ;
	\draw    (637.67,87.02) .. controls (639.11,93.34) and (643.42,91.88) .. (646.78,99.65) ;
	\draw    (646.78,99.65) .. controls (646.54,113.25) and (625.69,136.56) .. (619.94,141.9) ;
	\draw    (618.98,100.13) .. controls (619.46,109.36) and (630.48,120.05) .. (632.88,123.45) ;
	\draw    (638.63,128.79) .. controls (640.07,135.1) and (644.38,133.65) .. (647.74,141.42) ;
	\draw    (647.74,141.42) .. controls (647.5,155.02) and (629.04,169.59) .. (625.21,183.18) ;
	\draw    (619.94,141.9) .. controls (620.42,151.13) and (630.96,160.36) .. (633.36,163.76) ;
	\draw    (639.11,169.1) .. controls (641.99,171.53) and (644.86,173.96) .. (648.22,181.73) ;
	\draw    (648.22,181.73) .. controls (647.98,195.33) and (635.28,208.44) .. (623.29,220.58) ;
	\draw    (625.21,183.18) .. controls (625.69,192.41) and (633.84,198.73) .. (636.23,202.13) ;
	\draw    (641.99,207.47) .. controls (646.3,209.9) and (647.74,212.32) .. (651.09,220.1) ;
	\draw    (651.09,220.1) .. controls (650.85,233.69) and (629.52,253.12) .. (623.77,258.46) ;
	\draw    (623.29,220.58) .. controls (617.54,229.81) and (634.32,236.61) .. (636.71,240.01) ;
	\draw    (642.47,245.35) .. controls (643.9,251.66) and (648.22,250.21) .. (651.57,257.98) ;
	\draw    (651.57,257.98) .. controls (651.33,271.58) and (629.04,291.97) .. (623.29,297.32) ;
	\draw    (623.77,258.46) .. controls (624.25,267.69) and (633.84,275.46) .. (636.23,278.86) ;
	\draw    (641.99,284.2) .. controls (643.42,290.52) and (647.74,289.06) .. (651.09,296.83) ;
	\draw    (651.09,296.83) .. controls (650.85,310.43) and (630,329.86) .. (624.25,335.2) ;
	\draw    (623.77,296.83) .. controls (624.25,306.06) and (634.8,313.34) .. (637.19,316.74) ;
	\draw    (641.99,320.63) .. controls (643.42,326.94) and (647.74,325.49) .. (651.09,333.26) ;
	\draw    (646.78,43.8) .. controls (647.74,46.23) and (649.18,53.51) .. (645.82,62.74) ;
	\draw    (502.74,52.54) .. controls (517.84,37.48) and (533.42,37) .. (549.48,51.57) ;
	\draw    (566.01,51.08) .. controls (581.11,36.03) and (596.69,35.54) .. (612.75,50.11) ;
	\draw    (479.49,40.88) .. controls (466.55,46.71) and (462.48,88.96) .. (478.77,133.65) ;
	\draw    (502.74,52.54) .. controls (489.08,88.48) and (512.09,94.79) .. (506.58,134.13) ;
	\draw    (549.48,51.57) .. controls (556.67,77.79) and (543.24,85.56) .. (549.48,117.13) ;
	\draw    (566.01,51.08) .. controls (555.71,80.71) and (579.67,101.11) .. (576.8,118.1) ;
	\draw    (483.81,347.34) .. controls (514.48,392.99) and (631.44,375.02) .. (651.09,350.26) ;
	\draw    (651.09,333.26) .. controls (660.2,341.51) and (651.57,351.23) .. (651.09,350.26) ;
	\draw    (579.19,231.75) .. controls (595.97,259.92) and (559.54,260.41) .. (577.52,336.17) ;
	\draw    (509.45,335.2) .. controls (523.11,356.57) and (549,353.65) .. (556.19,334.23) ;
	\draw    (577.52,336.17) .. controls (591.18,357.54) and (617.06,354.63) .. (624.25,335.2) ;
	\draw    (551.39,232.24) .. controls (534.62,262.83) and (564.33,254.09) .. (556.19,334.23) ;
	\draw    (479.01,228.35) .. controls (496.75,276.92) and (455.53,259.43) .. (483.81,347.34) ;
	\draw    (505.38,226.89) .. controls (523.11,275.46) and (481.17,247.29) .. (509.45,335.2) ;
	\draw  [color={rgb, 255:red, 230; green, 22; blue, 22 }  ,draw opacity=1 ] (466.31,155.26) .. controls (466.31,141.18) and (477.58,129.76) .. (491.48,129.76) .. controls (505.37,129.76) and (516.64,141.18) .. (516.64,155.26) .. controls (516.64,169.34) and (505.37,180.76) .. (491.48,180.76) .. controls (477.58,180.76) and (466.31,169.34) .. (466.31,155.26) -- cycle ;
	\draw  [color={rgb, 255:red, 230; green, 22; blue, 22 }  ,draw opacity=1 ] (240,147.79) .. controls (240,133.71) and (251.27,122.29) .. (265.16,122.29) .. controls (279.06,122.29) and (290.33,133.71) .. (290.33,147.79) .. controls (290.33,161.87) and (279.06,173.28) .. (265.16,173.28) .. controls (251.27,173.28) and (240,161.87) .. (240,147.79) -- cycle ;
	
	\draw (52,388.9) node [anchor=north west][inner sep=0.75pt]    {$K_{2} =P( -2,3,7)$};
	\draw (311.5,389.4) node [anchor=north west][inner sep=0.75pt]    {$K_{1} =\ 7_{2}$};
	\draw (528,388.9) node [anchor=north west][inner sep=0.75pt]    {$K_{0} =\ L10a18$};
\end{tikzpicture}

\caption{A skein relation for $P(-2,3,7)$}
\label{fig:skein237}
\end{figure}
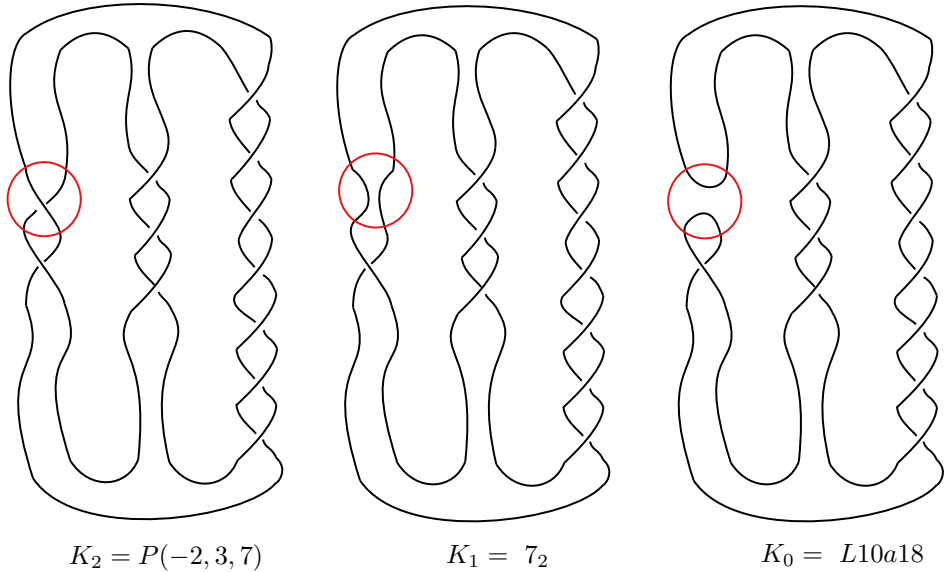

\section{Khovanov homology}
\label{subsec:odd_Kh}
The exposition of Khovanov homology in this section follows closely that of Baldwin-Hedden-Lobb~\cite{BHL2019}.

Let $D$ be a diagram of $K$ in $S^2 = \bbR^2 \cup \{\infty\}$ and label the crossings of $D$ as $\{1,\dots,N\}$.
For each crossing $j$, one can perform a $0$-resolution or a $1$-resolution as in Figure~\ref{fig:trianglemirror210}.
\begin{figure}[h!]
	\centering

\tikzset{every picture/.style={line width=0.75pt}} 

\begin{tikzpicture}[x=0.35pt,y=0.35pt,yscale=-1,xscale=1]
	
	\draw   (11.33,102.56) .. controls (11.33,52.17) and (52.57,11.33) .. (103.44,11.33) .. controls (154.32,11.33) and (195.56,52.17) .. (195.56,102.56) .. controls (195.56,152.94) and (154.32,193.78) .. (103.44,193.78) .. controls (52.57,193.78) and (11.33,152.94) .. (11.33,102.56) -- cycle ;
	\draw   (235.33,101.89) .. controls (235.33,51.51) and (276.57,10.67) .. (327.44,10.67) .. controls (378.32,10.67) and (419.56,51.51) .. (419.56,101.89) .. controls (419.56,152.27) and (378.32,193.11) .. (327.44,193.11) .. controls (276.57,193.11) and (235.33,152.27) .. (235.33,101.89) -- cycle ;
	\draw   (465.33,102.56) .. controls (465.33,52.17) and (506.57,11.33) .. (557.44,11.33) .. controls (608.32,11.33) and (649.56,52.17) .. (649.56,102.56) .. controls (649.56,152.94) and (608.32,193.78) .. (557.44,193.78) .. controls (506.57,193.78) and (465.33,152.94) .. (465.33,102.56) -- cycle ;
	\draw    (39.33,37.67) -- (90.89,96.44) ;
	\draw    (104.89,111.11) -- (162.22,172.44) ;
	\draw    (162.22,33.11) -- (40.89,167.11) ;
	\draw    (272.22,30.44) .. controls (328.22,71.11) and (287.56,177.11) .. (268.22,172.44) ;
	\draw    (375.56,24.44) .. controls (320.89,75.11) and (362.22,156.44) .. (387.56,171.78) ;
	\draw    (482,52) .. controls (518.22,93.11) and (597.56,94.44) .. (628.89,45.78) ;
	\draw    (484.22,157.78) .. controls (526.44,119.56) and (568.22,117.78) .. (632.89,154.44) ;
	
	\draw (86,201.4) node [anchor=north west][inner sep=0.75pt]    {$K_{-1}$};
	\draw (310,201.4) node [anchor=north west][inner sep=0.75pt]    {$K_{0}$};
	\draw (540,201.4) node [anchor=north west][inner sep=0.75pt]    {$K_{1}$};	
\end{tikzpicture}
\caption{An unoriented skein triple in Khovanov homology.}
\label{fig:trianglemirror210}
\end{figure}
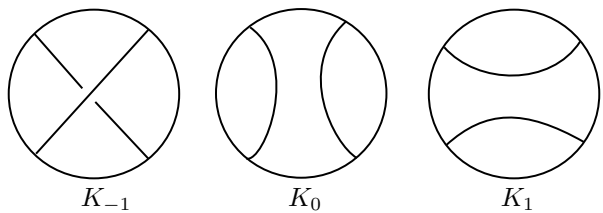
The different ways of resolving $D$ form a \emph{cube of resolutions}
parametrized by $\{0,1\}^N$: let $v_j$ be the $j$-th coordinate of $v \in \{0,1\}^N$, each vertex of the cube is a diagram $D_v$ obtained by doing $v_j$-resolution at the $j$-th crossing, for all $1 \le j \le N$.
Every $D_v$ consists of a disjoint union of circles.

Given $v \in \{0,1\}^N$, let $V(D_v)$ be the vector space generated by the components of $D_v$.
Consider the exterior algebra $\Lambda^*V(D_v)$ with a grading $\pgrade$ as follows.
The element $1 \in \Lambda^0V(D_v)$ is declared to have grading $\pgrade(1)$ equal to the number of components of $D_v$, and wedging with any of the components decreases $\pgrade$ by $2$.

We view $\{0,1\}^N$ as a subset of the lattice $\bbZ^N$.
The cube of resolution comes with a partial ordering: we write $u \le v$ if and only if $u_j \le v_j$, for all $j \in \{0,1\}$.
Moreover, we write $u <_k v$ if $u \le v$ and $|u-v|_1=k$.
Here, $|\cdot|_1$ is the $\ell^1$-norm on $\bbZ^N$, i.e. 
\[
	|w|_1 = \sum_{j=1}^N |w_j|.
\]
In other words, $v$ can be obtained from $u$ changing exactly $k$ $0$'s to $1$'s.
If $v,v'$ are such that $v <_1 v'$, we will define a map
\[
	d_{v,v'}: \Lambda^*V(D_v) \to \Lambda^*V(D_{v'}).
\]
The sum of all such maps
\begin{equation}
	\label{eqn:Kh_diff}
	d = \bigoplus_{v <_1 v'} d_{v,v'}
\end{equation}
form the differential of the \emph{Khovanov chain complex of $D$}:
\[
	\CKh(D) = \bigoplus_{v \in \{0,1\}^N} \Lambda^*V(D_v).
\]
Let $n_+,n_-$ be the number of the positive and negative crossings of $D$, respectively.
There is a (co-)homological grading $\hgrade$ on $\CKh(D)$, given by, for homogeneous $x \in \Lambda^*(D_v)$,
\[
	\hgrade(x) = v_1 + \dots + v_n - n_-,
\]
and a \emph{quantum grading $\qgrade$}:
\[
	\qgrade(x) = \pgrade(x) + \hgrade(x) + n_+ - n_-.
\]
In particular, $\CKh(D)$ is doubly graded, where the differential $d$ preserves $\qgrade$ and increases $\hgrade$ by $1$.
We write $\CKh^{i,j}(D)$ for the summand in homological grading $i$ and quantum grading $j$.
Then the \emph{Khovanov homology of $D$} is the bigraded vector space
\[
	\Kh(D) = \bigoplus_{i,j} \Kh^{i,j}(D),
\]
where $\Kh^{*,j}(D)$ is the homology of $d$ at the $j$th quantum grading, i.e.
\[
	\Kh^{i,j}(D) = H_i(\CKh^{*,j}(D),d).
\]
To define $d_{v,v'}$ when $v <_1 v'$, one considers two scenarios 
\begin{itemize}
	\itemindent=-13pt
\item (Merging) two circles $x$ and $y$ merges into one circle, or
\item (Splitting) a circle is splitted into two circles $x$ and $y$.
\end{itemize}
In the first scenario,
there is an identifcation
\[
	V(D_{v'}) \cong V(D_v)/(x+y),
\]
and we define the \emph{merge map} $d_{v,v'}$ to be the induced quotient map
\[
	\Lambda^*V(D_v) \to \Lambda^*(V(D_v)/(x+y)) \cong 
	\Lambda^*V(D_{v'}).
\]
In the latter scenario, we make use of the identification $V(D_v) \cong V(D_{v'})/(x+y)$, which induces an identification
\[
	\Lambda^*V(D_v) \cong 
	\Lambda^*(V(D_{v'})/(x+y)) \cong
	(x+y) \wedge \Lambda^*V(D_{v'}).
\]
The differential $d_{v,v'}$ is the \emph{split map}, defined as the composition
\[
	\Lambda^*V(D_v) \to
	\Lambda^*(V(D_{v'})/(x+y)) \to
	(x+y) \wedge \Lambda^*V(D_{v'})
	\hookrightarrow
	\Lambda^*V(D_{v'}).
\]
There is a  chain map
\[
	\Phi_{\infty}: \CKh(D) \to \CKh(D)
\]
given on each $\Lambda^*V(D_v)$ by wedging the component of $D_v$ containing $\infty$.
The \emph{reduced Khovanov complex} of $D$ is the quotient complex
\[
	\CKhr(D) 
	= (\CKh(D)/\Im(\Phi_{\infty}))[0,-1].
\]
The ``$[0,-1]$'' indicates a degree shift of $(0,-1)$ in the $(i,j)$ grading.

\section{A spectral sequence from Khovanov homology}
\label{sec:a_SS}
We begin with an overview of the algebraic structures that will appear in proofs of Theorem~\ref{thm:SS_intro} and Theorem~\ref{thm:triangle_intro}.

Let $K \subset S^3$ be a link and $D \subset S^2$ be a diagram for $K$.
Fix a base point $p$ and view it as $\infty \in S^2$.
Label the crossings as $\{1,\dots, N\}$ and consider the cube of resolutions that arises from the skein resolutions as in Figure~\ref{fig:triangle210}.
\begin{figure}[th!]
	\centering
\tikzset{every picture/.style={line width=0.75pt}} 

\begin{tikzpicture}[x=0.35pt,y=0.35pt,yscale=-1,xscale=1]
	
	\draw   (11.33,102.56) .. controls (11.33,52.17) and (52.57,11.33) .. (103.44,11.33) .. controls (154.32,11.33) and (195.56,52.17) .. (195.56,102.56) .. controls (195.56,152.94) and (154.32,193.78) .. (103.44,193.78) .. controls (52.57,193.78) and (11.33,152.94) .. (11.33,102.56) -- cycle ;
	\draw   (235.33,101.89) .. controls (235.33,51.51) and (276.57,10.67) .. (327.44,10.67) .. controls (378.32,10.67) and (419.56,51.51) .. (419.56,101.89) .. controls (419.56,152.27) and (378.32,193.11) .. (327.44,193.11) .. controls (276.57,193.11) and (235.33,152.27) .. (235.33,101.89) -- cycle ;
	\draw   (465.33,102.56) .. controls (465.33,52.17) and (506.57,11.33) .. (557.44,11.33) .. controls (608.32,11.33) and (649.56,52.17) .. (649.56,102.56) .. controls (649.56,152.94) and (608.32,193.78) .. (557.44,193.78) .. controls (506.57,193.78) and (465.33,152.94) .. (465.33,102.56) -- cycle ;
	\draw    (39.33,37.67) -- (166.89,167.11) ;
	\draw    (94.89,105.78) -- (40.22,167.11) ;
	\draw    (162.22,33.11) -- (107.56,94.44) ;
	\draw    (272.22,30.44) .. controls (328.22,71.11) and (287.56,177.11) .. (268.22,172.44) ;
	\draw    (375.56,24.44) .. controls (320.89,75.11) and (362.22,156.44) .. (387.56,171.78) ;
	\draw    (482,52) .. controls (518.22,93.11) and (597.56,94.44) .. (628.89,45.78) ;
	\draw    (484.22,157.78) .. controls (526.44,119.56) and (568.22,117.78) .. (632.89,154.44) ;
	
	\draw (88,202.07) node [anchor=north west][inner sep=0.75pt]    {$K_{2}$};
	\draw (321.33,202.07) node [anchor=north west][inner sep=0.75pt]    {$K_{1}$};
	\draw (551.33,203.4) node [anchor=north west][inner sep=0.75pt]    {$K_{0}$};
\end{tikzpicture}
\caption{An unoriented skein triple in $\HMR$.}
\label{fig:triangle210}
\end{figure}
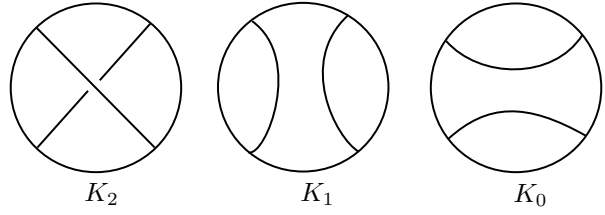
This figure is precisely the mirror of Figure~\ref{fig:trianglemirror210}, reflecting the fact that $\HMR$ is relate to Khovanov homology of the mirrors.
In particular, labelling of the resolution is now descending, where $K_2$ is by convention the trivial resolution. 
Let $v \in \{0,1,2\}^N$ and $v_j$ be its $j$th coordinate.
Denote by $D_v$ and $K_v$ the diagram and the link resulting from $v_j$-resolution at the $j$th crossing, for every $1 \le j \le N$, respectively.

We will define a chain complex
\[
	\bfC(D) = 
	\bigoplus_{v \in \{0,1\}^N} \widetilde{C}(K_v)
\]
and a differential 
\[\bfF = \sum_{v \ge u} f_{v,u}: \bfC \to \bfC\] 
having components
\[f_{v,u}:\widetilde{C}(K_v) \to \widetilde{C}(K_u).
\]
In this section, we will prove the following theorem, which generalizes the unoriented skein exact triangle in the case $N=1$.
\begin{thm}
\label{thm:quasi_iso_1}
	There is a quasi-isomorphism $(C_{w},f_{w,w}) \to (\bfC,\bfF)$, where $w=(2,\dots,2)$.
\end{thm}
Furthermore, $(\bfC,\bfF)$ carries a filtration $\cdots \subset \Fltr_n \subset \Fltr_{n+1} \subset \cdots$ by the sum of coordinates, preserved by the differential, where
\[
	\Fltr_n = \bigoplus_{\substack{{v \in \{0,1\}^N}\\{|v| \le n}}} \widetilde{C}(K_v).
\]
This filtration induces a spectral sequence, and by Theorem~\ref{thm:quasi_iso_1} the $E_{\infty}$ page can be identified with $\widetilde{\HMR}_{\bullet}$.

The $E_1$ page comprises $\widetilde{\HMR}(K_v)$ for all $v \in \{0,1\}^N$, which will be identified with $\CKhr(\overline{D}_v)$.
The differential $\bfF_1$ on the first page are formed by all edges $(v,v')$ such that $v >_1 v'$.
We will identify $\bfF_1$ with the Khovanov differential~\eqref{eqn:Kh_diff} of the mirror of $K$ in Proposition~\ref{prop:E1diso}.

\subsection{Topology of skein cobordisms}
\label{sec:top_skein_cob}
Fix once and for all a base point $p \in K$ playing the role of $\infty$ in $S^2 =  \bbR^2 \cup \{\infty\}$, away from all the skein resolutions.
Denote $S^3$ by $ Y$.
Typically a cylinder $I \times  Y$ will be denoted by $ W$.

We begin by reviewing the case of a single skein relation. 
Let $K_2,K_1,K_0$ be three links related as in Figure~\ref{fig:triangle210}.
From a 3-dimensional viewpoint, there is a $3$-ball $B^3$ outside which the $K_j$'s coincide. 
Each intersection of $K_j$ with the $B^3$ is formed by two arcs.
By identifying the $3$-ball with a solid regular tetrahedron, the arcs can be modelled as pairs of opposite edges, as in Figure~\ref{fig:arcintetrahedra}.
\begin{figure}[th!]
	\centering
\tikzset{every picture/.style={line width=0.75pt}}   
\begin{tikzpicture}[x=0.42pt,y=0.36pt,yscale=-1,xscale=1]
\draw [color={rgb, 255:red, 128; green, 128; blue, 128 }  ,draw opacity=1 ] [dash pattern={on 1.5pt off 1.5pt on 1.5pt off 1.5pt}]  (114.67,12.67) -- (206,127.67) ;
\draw [color={rgb, 255:red, 128; green, 128; blue, 128 }  ,draw opacity=1 ] [dash pattern={on 1.5pt off 1.5pt on 1.5pt off 1.5pt}]  (114.67,12.67) -- (38,128.33) ;
\draw [color={rgb, 255:red, 128; green, 128; blue, 128 }  ,draw opacity=1 ] [dash pattern={on 1.5pt off 1.5pt on 1.5pt off 1.5pt}]  (38,128.33) -- (206,127.67) ;
\draw [color={rgb, 255:red, 128; green, 128; blue, 128 }  ,draw opacity=1 ] [dash pattern={on 1.5pt off 1.5pt on 1.5pt off 1.5pt}]  (114.67,12.67) -- (116,218.33) ;
\draw [color={rgb, 255:red, 128; green, 128; blue, 128 }  ,draw opacity=1 ] [dash pattern={on 1.5pt off 1.5pt on 1.5pt off 1.5pt}]  (38,128.33) -- (116,218.33) ;
\draw [color={rgb, 255:red, 128; green, 128; blue, 128 }  ,draw opacity=1 ] [dash pattern={on 1.5pt off 1.5pt on 1.5pt off 1.5pt}]  (116,218.33) -- (206,127.67) ;
\draw [color={rgb, 255:red, 128; green, 128; blue, 128 }  ,draw opacity=1 ] [dash pattern={on 1.5pt off 1.5pt on 1.5pt off 1.5pt}]  (324,11.33) -- (415.33,126.33) ;
\draw [color={rgb, 255:red, 128; green, 128; blue, 128 }  ,draw opacity=1 ] [dash pattern={on 1.5pt off 1.5pt on 1.5pt off 1.5pt}]  (324,11.33) -- (247.33,127) ;
\draw [color={rgb, 255:red, 128; green, 128; blue, 128 }  ,draw opacity=1 ] [dash pattern={on 1.5pt off 1.5pt on 1.5pt off 1.5pt}]  (247.33,127) -- (415.33,126.33) ;
\draw [color={rgb, 255:red, 128; green, 128; blue, 128 }  ,draw opacity=1 ] [dash pattern={on 1.5pt off 1.5pt on 1.5pt off 1.5pt}]  (324,11.33) -- (325.33,217) ;
\draw [color={rgb, 255:red, 128; green, 128; blue, 128 }  ,draw opacity=1 ] [dash pattern={on 1.5pt off 1.5pt on 1.5pt off 1.5pt}]  (247.33,127) -- (325.33,217) ;
\draw [color={rgb, 255:red, 128; green, 128; blue, 128 }  ,draw opacity=1 ] [dash pattern={on 1.5pt off 1.5pt on 1.5pt off 1.5pt}]  (325.33,217) -- (415.33,126.33) ;
\draw [color={rgb, 255:red, 128; green, 128; blue, 128 }  ,draw opacity=1 ] [dash pattern={on 1.5pt off 1.5pt on 1.5pt off 1.5pt}]  (534,12.67) -- (625.33,127.67) ;
\draw [color={rgb, 255:red, 128; green, 128; blue, 128 }  ,draw opacity=1 ] [dash pattern={on 1.5pt off 1.5pt on 1.5pt off 1.5pt}]  (534,12.67) -- (457.33,128.33) ;
\draw [color={rgb, 255:red, 128; green, 128; blue, 128 }  ,draw opacity=1 ] [dash pattern={on 1.5pt off 1.5pt on 1.5pt off 1.5pt}]  (457.33,128.33) -- (625.33,127.67) ;
\draw [color={rgb, 255:red, 128; green, 128; blue, 128 }  ,draw opacity=1 ] [dash pattern={on 1.5pt off 1.5pt on 1.5pt off 1.5pt}]  (534,12.67) -- (535.33,218.33) ;
\draw [color={rgb, 255:red, 128; green, 128; blue, 128 }  ,draw opacity=1 ] [dash pattern={on 1.5pt off 1.5pt on 1.5pt off 1.5pt}]  (457.33,128.33) -- (535.33,218.33) ;
\draw [color={rgb, 255:red, 128; green, 128; blue, 128 }  ,draw opacity=1 ] [dash pattern={on 1.5pt off 1.5pt on 1.5pt off 1.5pt}]  (535.33,218.33) -- (625.33,127.67) ;
\draw    (38,128.33) .. controls (68.67,141) and (128.67,59.67) .. (114.67,12.67) ;
\draw    (116,218.33) .. controls (114.67,180.33) and (151.33,127.67) .. (202.67,127.67) ;
\draw    (247.33,127) .. controls (287.33,97) and (340.67,177) .. (325.33,217) ;
\draw    (324,11.33) .. controls (312,59) and (376.67,128.33) .. (415.33,126.33) ;
\draw    (534,12.67) .. controls (510,73) and (514,184.33) .. (535.33,218.33) ;
\draw    (457.33,128.33) .. controls (481.33,111) and (500.67,105) .. (512,103) ;
\draw    (524,102.33) .. controls (558.67,97.67) and (612.67,116.33) .. (621.33,127.67) ;
\end{tikzpicture}
	\caption{Standard model of $(B^3,B^3 \cap K_j)$}
	\label{fig:arcintetrahedra}
\end{figure}
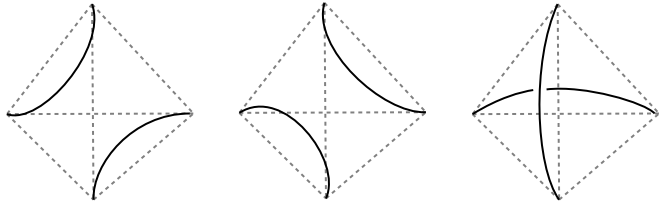
There is an order-$3$ symmetry in Figure~\ref{fig:arcintetrahedra}, realized by clockwise rotation of the tetrahedron, around the axis that from the top vertex to the centre of the bottom face.
We extend the indexing to $\bbZ$ with the understanding that $K_j = K_{j+3}$.

Each pair $(K_{j},K_{j-1})$ are related by a $1$-handle attachment cobordism $\Sigma_{j,j-1}$.
To be precise, $\Sigma_{j,j-1}$ is a properly embedded surface in $ W_{j,j-1} = [0,1] \times S^3$, bounding $\{0\} \times \overline{K_{j}}$ and $\{1\} \times K_{j-1}$.
Outside $[0,1] \times B^3$, we define $\Sigma_{j,j-1}$ as the product of $[0,1]$ with $K_j$.
Inside $[0,1] \times B^3$, we attach a saddle $T_{j,j-1}$ which is identified with a rectangle agreeing with $B^3 \cap K_j$ in $\{0\} \times B^3$ and $B^3 \cap K_{j-1}$ in $\{1\} \times B^3$.
This is demonstrated in Figure~\ref{fig:saddle1}
\begin{figure}[th!]
	\centering
	\includegraphics[width=0.45\linewidth]{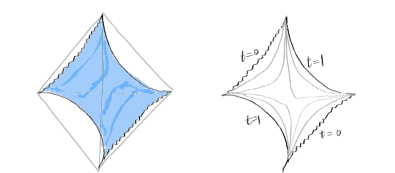}
	\caption{The $1$-handle $T_{j,j-1}$ as a rectangle.}
	\label{fig:saddle1}
\end{figure}

There is an arc $\delta_i \subset B^3$ lying on $\Sigma_{i,i-1}$ at time $\frac12$, connecting the two components of the link in $B^3$, see Figure~\ref{fig:delta_arc}. 
By taking the union of 
\begin{itemize}
	\itemindent=-13pt
	\item neighbourhoods (in $\Sigma_{i,i-2}$) of two arcs $\delta_{i+1}$ and $\delta_{i}$  at time $\frac12$ and $\frac32$, respectively, and
	\item two bands obtained from neighbourhoods of $(\delta \cap K_i) \subset K_i$, in time $[\frac12,\frac32]$,
\end{itemize}
we obtain a M\"obius band $M_{i,i-2} \subset [0,2] \times Y$.
This M\"obius band has self-intersection $(+2)$ (see \cite[Lemma~7.2]{KMunknot2011}).
Let $B_{i,i-2}$ be a regular neighbourhood (diffeomorphic to a $4$-ball) of $[\frac12,\frac32] \times \delta_{i,i-2}$ in $[\frac12,\frac32] \times B^3$.
Let $ S_{i,i-2}$ be the boundary of $ B_{i,i-2}$.
The hypersurface $S_{i,i-2} \subset  W_{i,i-2}$ intersects the M\"obius band $M_{i,i-1}$ at an unknot.
\begin{figure}[hbt]
	\centering
	\includegraphics[height=1.5in]{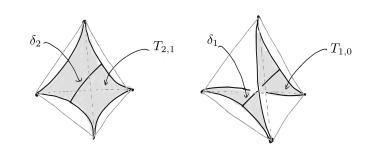}
	\caption{The arc $\delta_i \subset \Sigma_{i,i-1}$.}
	\label{fig:delta_arc}
\end{figure}

Inside the triple composition,
we glue the two M\"obius bands $M_{i,i-2}$ and $M_{i-1,i-3}$ to obtain $M_{i,i-3}$, which topologically is twice-punctured $\mathbb{RP}^2$.
\begin{figure}[hbt]
	\centering
	\includegraphics[height=1.5in]{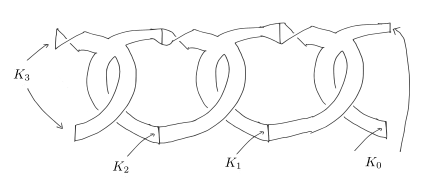}
	\caption{Gluing M\"obius bands.}
\end{figure}

Let $B_{i,i-3}$ be the regular neighbourhood of the union $B_{i,i-2} \cup B_{i-1,i-3}$ which is a $4$-ball containing $M_{i,i-3}$.
Let $R_{3,0}$ be the boundary $3$-sphere of $B_{3,0}$.
Then $R_{3,0}$ intersects the twice-punctured $\mathbb{RP}^2$ at a 2-component unlink. 
The following summarizes the topological observations that be found in \cite[Section~5]{KMOS2007} and \cite[Section~7]{KMunknot2011}.

\begin{prop}
\label{prop:double_cover_triangle}
The double branched cover $\bbB_{i,i-2}$ of $B_{i,i-2}$ is $\overline{\mathbb{CP}}^2$ minus a $4$-ball where the covering involution acts by complex conjugation.
The triple cobordism $(W_{3,0},\Sigma_{3,0})$ can be built from $([0,3] \times  Y) \setminus ( B_{3,0},\Delta)$ and $
(B_{3,0}, M_{3,0})$, glued along $( R_{3,0},\U_2)$.
In particular, $\Delta$ is a $2$-disc whose boundary coincides with an unlink $\U_2$.
In particular, the double branched cover $\bbN$ of $
(B_{3,0}, M_{3,0})$ is $(D^2 \times S^2) \# \overline{\mathbb{CP}}^2$.
A more symmetric description of $N$ is a $4$-manifold such that $\del N = S^1 \times S^2$ and contains two spheres $E_1,E_2$ satisfying $E_1 \cdot E_2 = 1$ and $E_i^2 = -1$. 
\hfill \qedsymbol
\end{prop}

Suppose now there are $N$ crossings, and for each
$v \in \{0,1,2\}^N$, let $K_v$ be the corresponding link in cube of resolutions.
Let $B_1,\dots,B_N$ be $3$-balls such that $K_v$'s agree outside $B_1 \cup \dots \cup B_N$, whereas inside $B_k$, each $K_v$ is modelled on the $v_j$'s picture in Figure~\ref{fig:arcintetrahedra}.
Using the 3-fold symmetry, we extend $\{0,1,2\}^N$ to $\bbZ^N$ periodically.
Equip $\bbZ^N$ with the partial order that $v \ge u$ if and only if $v_j \ge u_j$, for all $j$.
We consider the $\ell^{1}$ and $\ell^{\infty}$ norms on $\bbZ^N$:
\begin{align*}
	|v|_{1} &= \sum_{j=1}^N |v_j|, \\
	|v|_{\infty} &= \max_{1 \le j \le N} |v_j|.
\end{align*}
Whenever $v \ge u$ are such that $|v-u|_{\infty} \le 1$, we define the cobordism $\Sigma_{v,u} \subset [0,1] \times  Y$ from $K_v$ to $K_u$ the same way as before, in each ball where $v_j = u_j + 1$.
If $w \ge v \ge u$, where $|w - v|_{\infty} = |v-u|_{\infty} = 1$, we define concatenated cobordism $\Sigma_{w,u}$ to be $\Sigma_{v,u}\circ \Sigma_{w,v}$.
For $u \ge v$, the cobordism $\Sigma_{v,u}$ is understood as a surface in $[0,L] \times S^3$, where $L = |u-v|_{\infty}$.
We arrange so that the critical points of $\Sigma_{v,u}$ occur at $\{1/2,3/2,\dots,(L-\frac12)\}$.

Let $W^o$ be the $4$-manifold obtained from removing a small $4$-ball $B_o$ around $\{\frac14\} \times \{x\}$ from $[0,L] \times Y$. 
In particular, the puncture is on the left to all $1$-handles. 
The ball $B_o$ intersects $\Sigma_{v,u}$ at a small $2$-disc and the remaining surface $\Sigma_{v,u}^o$ is a properly embedded surface in $ W^{o}$.
We put a partial order on the lattice $\bbZ^N \times \{0,1\}$ by viewing it as a sublattice of $\bbZ^{N+1}$.
An element of $\bbZ^N \times \{0,1\}$ will be denoted as $w \in  \bbZ^{N+1}$, where $w_{N+1} \in \{0,1\}$.

We will attach cylindrical ends $(-\infty,0] \times ( Y, K_v)$ and $[L,\infty) \times ( Y, K_u)$ to $([0,L] \times  Y,\Sigma_{v,u})$.
By abusing notations, also $\Sigma_{v,u}$ denotes a compact surface with boundary, or a noncompact surface in the cylinder $(-\infty,+\infty) \times  Y$ or $((-\infty,+\infty) \times  Y) \setminus  B_o$ with infinite cylindrical ends.

\subsection{Hypersurfaces}
\begin{notat}
	We will use boldface number $\mathbf{n}$ to denote the vector $(n,\dots,n)$ in $ \bbZ^N$.
\end{notat}
Let $(W,\Sigma)$ and $(W^o,\Sigma^o)$ be the cobordisms from $\bfone$ to $\bfzero$, that is, for $v = (1,\dots,1)$ to $u=(0,\dots,0)$.
We have the analogue of \cite[Lemma~2.2]{Bloom2011}.
\begin{defn-prop}
	\label{prop:hypsurfaces1}
There exist $(2^N-2)$ hypersurfaces $ Y_v \cong S^3$, for $v \in \bbZ^N$ such that $\bfone > v  > \bfzero$.
Each $Y_v \cap \Sigma$ is a link $K_v$, and the intersection of $ Y_v$ and $Y_{v'}$ is a $2$-sphere if $v,v'$ are unordered. 
Furthermore, there exist $(2 \times 2^N-2)$ hypersurfaces $ Y_{w}\cong S^3$, where $w \in \bbZ^{N+1}$ such that $\bfzero 0 < w < \bfone 1$.
Let $ W^o_{w,w'}$ be the submanifold of $ W^o$ bounded by $ Y_{w}$ and $ Y_{w}$.
The pair $( Y_{w},  Y_{w} \cap \Sigma)$ is diffeomorphic to $(S^3, K_w)$; if $w > w'$, then the pair $( W_{w,w'},  W_{w,w'} \cap \Sigma^o)$ is diffeomorphic to $(I \times S^3, \Sigma_{ww'})$ if $w_{N+1}=w'_{N+1}$, and isomorphic to
$(I \times S^3 \setminus  B_o, \Sigma_{ww'}^o)$ otherwise.
\end{defn-prop} 
\begin{proof}
	The definition of hypersurfaces can be extracted from the proof of Proposition~\ref{prop:hypsurfaces2} and~\ref{prop:hypersurfaces2o}.
	For the hypersurfaces when $N=2$, see Figure~\ref{fig:prop:hypsurfaces1}.
	The additional hypersurfaces from the first statement to the second statement are red.
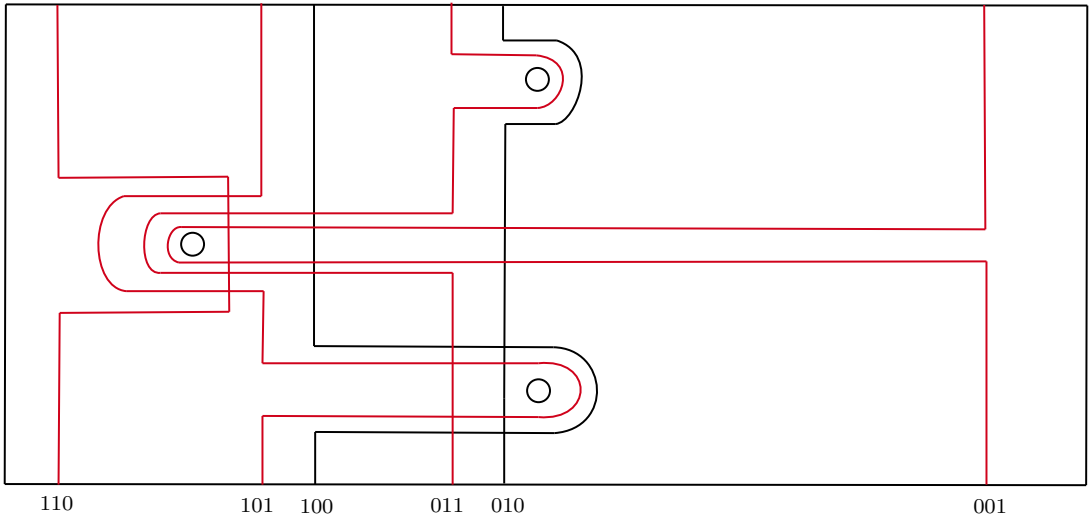
\begin{figure}[hbt]
\centering
\tikzset{every picture/.style={line width=0.75pt}} 
\begin{tikzpicture}[x=0.65pt,y=0.65pt,yscale=-1,xscale=1]
\draw    (10.67,10.33) -- (640,11) ;
\draw    (10,289.67) -- (639.33,290.33) ;
\draw   (313.33,54) .. controls (313.33,50.32) and (316.32,47.33) .. (320,47.33) .. controls (323.68,47.33) and (326.67,50.32) .. (326.67,54) .. controls (326.67,57.68) and (323.68,60.67) .. (320,60.67) .. controls (316.32,60.67) and (313.33,57.68) .. (313.33,54) -- cycle ;
\draw   (314,235.33) .. controls (314,231.65) and (316.98,228.67) .. (320.67,228.67) .. controls (324.35,228.67) and (327.33,231.65) .. (327.33,235.33) .. controls (327.33,239.02) and (324.35,242) .. (320.67,242) .. controls (316.98,242) and (314,239.02) .. (314,235.33) -- cycle ;
\draw    (640,11) -- (639.33,290.33) ;
\draw    (300,11.33) -- (300,31.33) ;
\draw    (301.33,80) -- (300.67,240) ;
\draw    (300.67,240) -- (300.67,289.33) ;
\draw    (190,10.33) -- (190,209.33) ;
\draw    (190.67,259.33) -- (190.67,290) ;
\draw    (301.33,80) -- (330.67,80) ;
\draw    (190,209.33) -- (329.33,210) ;
\draw    (190.67,259.33) -- (330,260) ;
\draw    (331.33,31.33) .. controls (357.33,40) and (342.67,78.67) .. (330.67,80) ;
\draw    (10.67,10.33) -- (10,289.67) ;
\draw   (112.67,150) .. controls (112.67,146.32) and (115.65,143.33) .. (119.33,143.33) .. controls (123.02,143.33) and (126,146.32) .. (126,150) .. controls (126,153.68) and (123.02,156.67) .. (119.33,156.67) .. controls (115.65,156.67) and (112.67,153.68) .. (112.67,150) -- cycle ;
\draw [color={rgb, 255:red, 208; green, 2; blue, 27 }  ,draw opacity=1 ]   (270,9.33) -- (270,39.33) ;
\draw [color={rgb, 255:red, 208; green, 2; blue, 27 }  ,draw opacity=1 ]   (270,39.33) -- (319.33,40) ;
\draw [color={rgb, 255:red, 208; green, 2; blue, 27 }  ,draw opacity=1 ]   (271.33,70.67) -- (270.67,132) ;
\draw [color={rgb, 255:red, 208; green, 2; blue, 27 }  ,draw opacity=1 ]   (319.33,40) .. controls (344.67,42.67) and (334.67,70) .. (320,70.67) ;
\draw [color={rgb, 255:red, 208; green, 2; blue, 27 }  ,draw opacity=1 ]   (270.67,166.67) -- (270.67,290) ;
\draw [color={rgb, 255:red, 208; green, 2; blue, 27 }  ,draw opacity=1 ]   (100.67,132) -- (270.67,132) ;
\draw [color={rgb, 255:red, 208; green, 2; blue, 27 }  ,draw opacity=1 ]   (100.67,166.67) -- (270.67,166.67) ;
\draw [color={rgb, 255:red, 208; green, 2; blue, 27 }  ,draw opacity=1 ]   (100.67,132) .. controls (88.67,132.67) and (87.33,168) .. (100.67,166.67) ;
\draw [color={rgb, 255:red, 208; green, 2; blue, 27 }  ,draw opacity=1 ]   (159.33,9.67) -- (159.33,122) ;
\draw [color={rgb, 255:red, 208; green, 2; blue, 27 }  ,draw opacity=1 ]   (160.67,177.33) -- (160,219.33) ;
\draw [color={rgb, 255:red, 208; green, 2; blue, 27 }  ,draw opacity=1 ]   (79.33,122) -- (159.33,122) ;
\draw [color={rgb, 255:red, 208; green, 2; blue, 27 }  ,draw opacity=1 ]   (80.67,177.33) -- (160.67,177.33) ;
\draw [color={rgb, 255:red, 208; green, 2; blue, 27 }  ,draw opacity=1 ]   (79.33,122) .. controls (58.67,128.67) and (60,175.33) .. (80.67,177.33) ;
\draw [color={rgb, 255:red, 208; green, 2; blue, 27 }  ,draw opacity=1 ]   (160,250) -- (160,290) ;
\draw [color={rgb, 255:red, 208; green, 2; blue, 27 }  ,draw opacity=1 ]   (160,250) -- (320.67,250.67) ;
\draw [color={rgb, 255:red, 208; green, 2; blue, 27 }  ,draw opacity=1 ]   (160,219.33) -- (320.67,219.33) ;
\draw  [color={rgb, 255:red, 208; green, 2; blue, 27 }  ,draw opacity=1 ]  (40.67,10.67) -- (41.33,111.33) ;
\draw   [color={rgb, 255:red, 208; green, 2; blue, 27 }  ,draw opacity=1 ] (42,190) -- (41.33,290) ;
\draw  [color={rgb, 255:red, 208; green, 2; blue, 27 }  ,draw opacity=1 ]  (41.33,111.33) -- (140,110.67) ;
\draw  [color={rgb, 255:red, 208; green, 2; blue, 27 }  ,draw opacity=1 ] (42,190) -- (140.67,189.33) ;
\draw  [color={rgb, 255:red, 208; green, 2; blue, 27 }  ,draw opacity=1 ]  (140,110.67) -- (140.67,189.33) ;
\draw [color={rgb, 255:red, 208; green, 2; blue, 27 }  ,draw opacity=1 ]   (580,10.67) -- (580.67,141.33) ;
\draw [color={rgb, 255:red, 208; green, 2; blue, 27 }  ,draw opacity=1 ]   (581.33,160) -- (581.33,290) ;
\draw [color={rgb, 255:red, 208; green, 2; blue, 27 }  ,draw opacity=1 ]   (112.67,140) -- (580.67,141.33) ;
\draw [color={rgb, 255:red, 208; green, 2; blue, 27 }  ,draw opacity=1 ]   (111.33,160.67) -- (581.33,160) ;
\draw [color={rgb, 255:red, 208; green, 2; blue, 27 }  ,draw opacity=1 ]   (112.67,140) .. controls (104.67,139.33) and (100.67,158) .. (111.33,160.67) ;
\draw [color={rgb, 255:red, 208; green, 2; blue, 27 }  ,draw opacity=1 ]   (271.33,70.67) -- (320,70.67) ;
\draw    (329.33,210) .. controls (361.33,210.67) and (364.67,257.33) .. (330,260) ;
\draw [color={rgb, 255:red, 208; green, 2; blue, 27 }  ,draw opacity=1 ]   (320.67,219.33) .. controls (353.33,216) and (353.33,253.33) .. (320.67,250.67) ;
\draw    (300,31.33) -- (331.33,31.33) ;
\draw (357.33,-43) node [anchor=north west][inner sep=0.75pt]   [align=left] {};
\draw (572,296.73) node [anchor=north west][inner sep=0.75pt]  [font=\footnotesize]  {$001$};
\draw (291.33,296.07) node [anchor=north west][inner sep=0.75pt]  [font=\footnotesize]  {$010$};
\draw (256,296.07) node [anchor=north west][inner sep=0.75pt]  [font=\footnotesize]  {$011$};
\draw (180,296.73) node [anchor=north west][inner sep=0.75pt]  [font=\footnotesize]  {$100$};
\draw (145.33,296.07) node [anchor=north west][inner sep=0.75pt]  [font=\footnotesize]  {$101$};
\draw (28.67,295.4) node [anchor=north west][inner sep=0.75pt]  [font=\footnotesize]  {$110$};
\end{tikzpicture}
\caption{Hypersurfaces in Definition-Proposition~\ref{prop:hypsurfaces1} when $N=2$.}
\label{fig:prop:hypsurfaces1}
\end{figure} 
\end{proof}
Let $( W,\Sigma)$ be the cobordism from $ \bftwo$ to $\bfzero$.
We need $N$ new auxiliary hypersurfaces.
\begin{defn-prop}
	\label{prop:hypsurfaces2}
There exist $(3^N-2)$ hypersurfaces $Y_v \cong S^3$, for $v \in \bbZ^N$ such that $\bftwo > v > \bfzero$.
In addition, there exist $N$ pairwise disjoint hypersurfaces $S_k \cong S^3$ for $k \in \{1,\dots,N\}$.
These hypersurfaces satisfy the following.
\begin{itemize}
	\itemindent=-13pt
	\item $Y_{v^1},\dots, Y_{v^m}$ are pairwise disjoint if $v^1 < v^2 < \dots < v^m$.
	\item The intersection of $Y_v$ and $Y_{v'}$ consists of $2$-spheres (possibly empty).
	\item The intersection of $Y_v$ and $S_k$ is nonempety if and only if $v_k = 1$. 
\end{itemize}
Each $ Y_v \cap \Sigma$ is the link $K_v$, and $S_k \cap \Sigma$ is a unknot which bounds a M\"{o}bius band.
If $v > v'$, then $Y_{v},Y_{v'}$ together bound a submanifold $W_{v,v'}$.
The pair $(W_{v,v'},\Sigma \cap W_{v,v'})$ is diffeomorphic to $(I \times S^3, \Sigma_{v,v'})$.
\end{defn-prop} 
\begin{proof}
The proof is analogous to \cite[Proposition~5.1]{Bloom2011}, except in reverse order.
We totally order the set of $v$'s first by descreasing $|v|_1$ and then (decreasing) numerically within the same $|\cdot|_1$. 
Choose a sequence $t_v \in (\frac14,2)$ in an order-preserving fashion.
On the complement of $[0,2] \times (Y \setminus \bigcup B_k)$ in $W$, the hypersurface $ Y_{v}$ is a slice 
	\[
		\{t_{v}\} \times (Y \setminus \bigcup_{k=1}^N B_k).
	\]
If $v_k=2$, we glue back in the ball $\{t_v\} \times B_k$.
If $v_k =0$ or $1$,
	we glue in the $t > t_v$ half of of the boundary of a neighbourhood of the cylinder 
	\[
	\left[t_v, (2-v_k)-\frac12
	\right] \times B_k.
	\]
(Recall that $((2-v_k)-\frac12)$ is where the critical points occur.)
For a given $k$, we arrange such neighbourhoods so that their size increase respect to the ordering.
As a result, the resulting $Y_{v}$ is to the right of $Y_{v'}$ if $t_v > t_{v'}$.
We can assume also that the maximum $t$-coordinates of all $ Y_{v}$ having the same $|v|_{\infty}$ is less than $t_{v'}$ for all $|v'|_{\infty}>|v|_{\infty}$.
The configuration when $N=2$ is illustrated in Figure~\ref{fig:prop:hypsurfaces2}.
\begin{figure}
\centering
\tikzset{every picture/.style={line width=0.75pt}} 
\begin{tikzpicture}[x=0.65pt,y=0.65pt,yscale=-1,xscale=1]
\draw    (10.67,10.33) -- (640,11) ;
\draw    (10,289.67) -- (639.33,290.33) ;
\draw  [dash pattern={on 0.84pt off 2.51pt}]  (320,10.33) -- (320.67,290.33) ;
\draw  [fill={rgb, 255:red, 255; green, 255; blue, 255 }  ,fill opacity=1 ] (157.33,239) .. controls (157.33,237.16) and (158.83,235.67) .. (160.67,235.67) .. controls (162.51,235.67) and (164,237.16) .. (164,239) .. controls (164,240.84) and (162.51,242.33) .. (160.67,242.33) .. controls (158.83,242.33) and (157.33,240.84) .. (157.33,239) -- cycle ;
\draw [color={rgb, 255:red, 208; green, 2; blue, 27 }  ,draw opacity=1 ]   (161.33,221) -- (490.67,220.33) ;
\draw [color={rgb, 255:red, 208; green, 2; blue, 27 }  ,draw opacity=1 ]   (160.67,253) -- (490,252.33) ;
\draw [color={rgb, 255:red, 208; green, 2; blue, 27 }  ,draw opacity=1 ]   (160.67,253) .. controls (145.33,247) and (150.67,223) .. (161.33,221) ;
\draw [color={rgb, 255:red, 208; green, 2; blue, 27 }  ,draw opacity=1 ]   (490.67,220.33) .. controls (512,219.67) and (513.33,251.67) .. (490,252.33) ;
\draw [color={rgb, 255:red, 208; green, 2; blue, 27 }  ,draw opacity=1 ]   (161.33,44.33) -- (490.67,43.67) ;
\draw [color={rgb, 255:red, 208; green, 2; blue, 27 }  ,draw opacity=1 ]   (162,73.67) -- (491.33,73) ;
\draw [color={rgb, 255:red, 208; green, 2; blue, 27 }  ,draw opacity=1 ]   (162,73.67) .. controls (146.67,67.67) and (150.67,46.33) .. (161.33,44.33) ;
\draw [color={rgb, 255:red, 208; green, 2; blue, 27 }  ,draw opacity=1 ]   (490.67,43.67) .. controls (512,43) and (514.67,72.33) .. (491.33,73) ;
\draw    (470.67,11) -- (470.67,21) ;
\draw    (470.67,100.33) -- (470.67,201) ;
\draw    (470.67,201) -- (470,290.33) ;
\draw    (500.67,21) .. controls (580.67,20.33) and (581.33,99) .. (500.67,100.33) ;
\draw    (441.33,11.67) -- (440.67,200.33) ;
\draw    (441.33,279) -- (441.33,291) ;
\draw    (536,200.33) .. controls (582,201) and (581.33,279.67) .. (536.67,279) ;
\draw    (399.33,11.67) -- (399.33,31) ;
\draw    (400.67,84.33) -- (401.33,197) ;
\draw    (400.67,271) -- (400.67,291) ;
\draw    (399.33,31) -- (478.67,31.67) ;
\draw    (400.67,84.33) -- (479.33,83.67) ;
\draw    (496.67,31.67) .. controls (570,31.67) and (555.33,83.67) .. (497.33,83.67) ;
\draw    (147.33,197.67) -- (401.33,197) ;
\draw    (151.33,270.33) -- (400.67,271) ;
\draw    (147.33,197.67) .. controls (134.67,211.67) and (132.67,249.67) .. (151.33,270.33) ;
\draw    (320,10.33) -- (320.67,290.33) ;
\draw    (270,11) -- (270,20.33) ;
\draw    (271.33,204.33) -- (490.67,205.67) ;
\draw    (271.33,263) -- (271.33,289) ;
\draw    (271.33,263) -- (492,263.67) ;
\draw    (510,206.33) .. controls (546.67,207.67) and (547.33,263.67) .. (511.33,264.33) ;
\draw    (109.33,20) -- (270,20.33) ;
\draw    (270.33,100.33) -- (271.33,204.33) ;
\draw    (110,100.67) -- (270.67,101) ;
\draw    (109.33,20) .. controls (75.33,34.33) and (79.33,85.67) .. (110,100.67) ;
\draw    (120.67,11.33) -- (120.67,30.33) ;
\draw    (120,89.67) -- (120,290.33) ;
\draw    (120.67,30.33) -- (169.33,30.33) ;
\draw    (70.67,11.33) -- (71.33,210.33) ;
\draw    (120,89.67) -- (171.33,90.33) ;
\draw    (169.33,30.33) .. controls (220.67,30.33) and (214.67,88.33) .. (171.33,90.33) ;
\draw    (71.33,210.33) -- (172,210.33) ;
\draw    (68,260.33) -- (170.67,261) ;
\draw    (68,260.33) -- (68,289) ;
\draw    (172,210.33) .. controls (207.33,211.67) and (207.33,261) .. (170.67,261) ;
\draw    (10.67,10.33) -- (10,289.67) ;
\draw    (640,11) -- (639.33,290.33) ;
\draw    (470.67,21) -- (500.67,21) ;
\draw    (470.67,100.33) -- (500.67,100.33) ;
\draw    (478.67,31.67) -- (496.67,31.67) ;
\draw    (479.33,83.67) -- (497.33,83.67) ;
\draw    (440.67,200.33) -- (536,200.33) ;
\draw    (441.33,279) -- (536.67,279) ;
\draw    (490.67,205.67) -- (510,206.33) ;
\draw    (492,263.67) -- (511.33,264.33) ;
\draw  [fill={rgb, 255:red, 255; green, 255; blue, 255 }  ,fill opacity=1 ] (157.33,58.33) .. controls (157.33,56.49) and (158.83,55) .. (160.67,55) .. controls (162.51,55) and (164,56.49) .. (164,58.33) .. controls (164,60.17) and (162.51,61.67) .. (160.67,61.67) .. controls (158.83,61.67) and (157.33,60.17) .. (157.33,58.33) -- cycle ;
\draw  [fill={rgb, 255:red, 255; green, 255; blue, 255 }  ,fill opacity=1 ] (486.67,60.33) .. controls (486.67,58.49) and (488.16,57) .. (490,57) .. controls (491.84,57) and (493.33,58.49) .. (493.33,60.33) .. controls (493.33,62.17) and (491.84,63.67) .. (490,63.67) .. controls (488.16,63.67) and (486.67,62.17) .. (486.67,60.33) -- cycle ;
\draw  [fill={rgb, 255:red, 255; green, 255; blue, 255 }  ,fill opacity=1 ] (486.67,237.83) .. controls (486.67,235.99) and (488.16,234.5) .. (490,234.5) .. controls (491.84,234.5) and (493.33,235.99) .. (493.33,237.83) .. controls (493.33,239.67) and (491.84,241.17) .. (490,241.17) .. controls (488.16,241.17) and (486.67,239.67) .. (486.67,237.83) -- cycle ;

\draw (622.67,293.73) node [anchor=north west][inner sep=0.75pt]  [font=\footnotesize]  {$00$};
\draw (464,293.73) node [anchor=north west][inner sep=0.75pt]  [font=\footnotesize]  {$01$};
\draw (432.67,293.4) node [anchor=north west][inner sep=0.75pt]  [font=\footnotesize]  {$10$};
\draw (392.67,293.4) node [anchor=north west][inner sep=0.75pt]  [font=\footnotesize]  {$02$};
\draw (312.67,294.73) node [anchor=north west][inner sep=0.75pt]  [font=\footnotesize]  {$11$};
\draw (263.33,294.07) node [anchor=north west][inner sep=0.75pt]  [font=\footnotesize]  {$20$};
\draw (111.33,293.4) node [anchor=north west][inner sep=0.75pt]  [font=\footnotesize]  {$12$};
\draw (60,293.4) node [anchor=north west][inner sep=0.75pt]  [font=\footnotesize]  {$21$};
\draw (12,293.07) node [anchor=north west][inner sep=0.75pt]  [font=\footnotesize]  {$22$};
\draw (350.67,23.4) node [anchor=north west][inner sep=0.75pt]    {$\textcolor[rgb]{0.82,0.01,0.11}{S_{1}}$};
\draw (351.33,232.07) node [anchor=north west][inner sep=0.75pt]    {$\textcolor[rgb]{0.82,0.01,0.11}{S}\textcolor[rgb]{0.82,0.01,0.11}{_{2}}$};
\end{tikzpicture}
\caption{Hypersurfaces in Definition-Proposition~\ref{prop:hypsurfaces2} when $N=2$.}
\label{fig:prop:hypsurfaces2}
\end{figure}
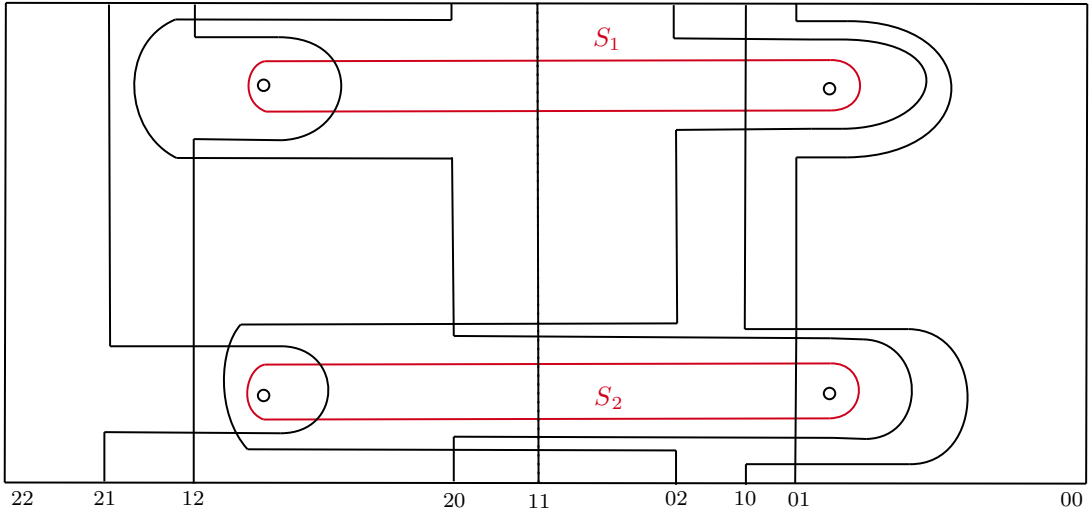 

To construct $ S_k$, recall in the discussion on a single crossing that is a $4$-ball defined as the regular neighbourhood of 
\[
	\left[\frac12, \frac32\right] \times \delta, 
\]
where $\delta$ is the corresponding core of of the two dimensional $1$-handle in $\Sigma$.
Let $ S_k$ be the resulting boundary $3$-sphere, arranged to not to intersect any $ Y_v$, for $|v|_{\infty} = 2$.
In particular, $ S_k \cap \Sigma$ is an unknot and bounds a M\"{o}bius band.

By construction, a pair of hypersurfaces $ Y_v, Y_{v'}$ are disjoint if they are ordered, and intersect at a $2$-sphere if not ordered.
Moreover, since $ S_k$ contains both critical points in the $B_k$ ball, $ Y_v \cap  S_k \ne \emptyset$ if and only if $v_k = 1$.
As a final remark, the double branched covers of $\{ Y_v,  S_k\}$ along $\Sigma$ satisfy exactly the same properties as \cite[Proposition~5.1]{Bloom2011}. 
\end{proof}

We modify the hypersurfaces in Proposition~\ref{prop:hypsurfaces2} as follows.
\begin{defn-prop}
	\label{prop:hypersurfaces2o}
	Within $ W^o = W_{\bftwo,\bfzero}^o$, there exist $(2 \times 3^N-2)$ hypersurfaces $ Y_{w} \cong S^3$, for $w \in  \bbZ^{N+1}$ such that $\bfzero < w < \bftwo$.
	There are $N$ pairwise disjoint hypersurfaces $ S_k \cong S^3$.
	Each $ Y_{w} \cap \Sigma$ is a link $K_w$, and $ S_k \cap \Sigma$ is a unknot bounding a M\"{o}bius band.
	These hypersurfaces satisfy the following.
	\begin{itemize}
		\itemindent=-13pt
		\item $ Y_{w^1},\dots,  Y_{w^m}$ are pairwise disjoint if $w^1 < w^2 < \dots < w^m$.
		\item The intersection of $ Y_v$ and $ Y_u$ consists of $2$-spheres, possibly empty.
		\item The intersection of $ Y_v$ and $ S_k$ is nonempety if and only if $v_k = 1$. 
	\end{itemize}
Each $ Y_{w} \cap \Sigma$ is a link $K_w$, where $K_w \subset  Y_{w}$ is isomorphic to the link corresponding to the first $N$ coordinates of $w$.
Moreover, $ S_k \cap \Sigma$ is a unknot bounding a M\"{o}bius band.
Suppose $w > w'$. 
Let $ W^o_{w,w'}$ be the submanifold bounded by $ Y_{w}$ and $ Y_{w'}$.
Then the pair $( W^o_{w,w'},  W^o_{w,w'} \cap \Sigma^o)$ is isomorphic to  $(S^3 \times I, \Sigma_{w,w'})$ if $w_{N+1} = w'_{N+1}$, and isomorphic to  $((S^3 \times I)\setminus  B_o, \Sigma^o_{w,w'})$ otherwise.
\end{defn-prop} 
\begin{proof}
	For this proposition, we need to modify the family of $\{ Y_v\}$ to a new family $\{ Y_{w}\}$, where $w = v*$, for $* \in \{0,1\}$.
	Suppose $v$'s are totally ordered as in Proposition~\ref{prop:hypsurfaces2} and $t_v$ are chosen. 
	We totally order $\{w\}$ by asserting $v1$ is less than $v0$.
	Define $t_{v0}=t_v$ and $t_{v1} = t_v - \epsilon$, where $\epsilon > 0$ is sufficiently small.
	Let $Y_{v0} = Y_v$.
	Over $[0,2]\times ( Y \setminus B_o)$, we take $ Y_{v1}$ to be a $(-\epsilon)$-translated copy of $Y_v$, slightly shrank on the $B_k$ regions.
	We choose a sequence of neighbourhoods labelled by $v$, whose size decreases with respect to the order, of 
	\[
		\left[\frac14-\epsilon,t_{v1}\right] \times B_o.
	\]
	We attach the $t < t_{v1}$ portion of the boundary of the neighbourhood to $ Y_{v0}$ (smooth and cut the corners if necessary).
	The new hypersurfaces that goes to the left of $B_o$     is similar to Figure~\ref{fig:prop:hypsurfaces1}.
	The double covers of our picture recovers that of \cite[Section~8]{Bloom2011}.
\end{proof}

We end this section by going back to the case of a single crossing.
\begin{defn-prop}
	\label{prop:hypersurfaces3}
	Let $( W,\Sigma)$ be the cobordism from $( Y,K_3)$ to $( Y,K_0)$.
	Let $( W^o,\Sigma^o)$ be the associated punctured cobordism.
	There are nine hypersurfaces $ Y_{30}$, $ Y_{21}$, $ Y_{20}$, $ Y_{11}$,  $ Y_{10}$, $ Y_{01}$, $ S_{3,1}$, $ S_{2,0}$, and $ R_{3,0}$ in $( W^o,\Sigma^o)$.
	At most three of the nine hypersurfaces can be mutually disjoint.
	\hfill \qedsymbol
\end{defn-prop}
Figure~\ref{fig:prop:hypersurfaces3} demonstrates the nine hypersurfaces in Proposition-Definition~\ref{prop:hypersurfaces3}.
\begin{figure}
\tikzset{every picture/.style={line width=0.75pt}} 
\begin{tikzpicture}[x=0.65pt,y=0.65pt,yscale=-1,xscale=1]
\draw    (10.67,10.33) -- (650,10) ;
\draw    (10,290) -- (650,290) ;
\draw    (10.67,10.33) -- (10,290) ;
\draw    (650,10) -- (650,290) ;
\draw  [fill={rgb, 255:red, 255; green, 255; blue, 255 }  ,fill opacity=1 ] (130,103.33) .. controls (130,101.49) and (131.49,100) .. (133.33,100) .. controls (135.17,100) and (136.67,101.49) .. (136.67,103.33) .. controls (136.67,105.17) and (135.17,106.67) .. (133.33,106.67) .. controls (131.49,106.67) and (130,105.17) .. (130,103.33) -- cycle ;
\draw  [fill={rgb, 255:red, 255; green, 255; blue, 255 }  ,fill opacity=1 ] (333.33,103.33) .. controls (333.33,101.49) and (334.83,100) .. (336.67,100) .. controls (338.51,100) and (340,101.49) .. (340,103.33) .. controls (340,105.17) and (338.51,106.67) .. (336.67,106.67) .. controls (334.83,106.67) and (333.33,105.17) .. (333.33,103.33) -- cycle ;
\draw  [fill={rgb, 255:red, 255; green, 255; blue, 255 }  ,fill opacity=1 ] (543.33,103.33) .. controls (543.33,101.49) and (544.83,100) .. (546.67,100) .. controls (548.51,100) and (550,101.49) .. (550,103.33) .. controls (550,105.17) and (548.51,106.67) .. (546.67,106.67) .. controls (544.83,106.67) and (543.33,105.17) .. (543.33,103.33) -- cycle ;
\draw  [color={rgb, 255:red, 0; green, 0; blue, 0 }  ,draw opacity=1 ][fill={rgb, 255:red, 0; green, 0; blue, 0 }  ,fill opacity=1 ] (70,226.67) .. controls (70,224.83) and (71.49,223.33) .. (73.33,223.33) .. controls (75.17,223.33) and (76.67,224.83) .. (76.67,226.67) .. controls (76.67,228.51) and (75.17,230) .. (73.33,230) .. controls (71.49,230) and (70,228.51) .. (70,226.67) -- cycle ;
\draw    (460,10) -- (460,290) ;
\draw    (230,10) -- (230,290) ;
\draw   (70,100) .. controls (70,61.34) and (193.12,30) .. (345,30) .. controls (496.88,30) and (620,61.34) .. (620,100) .. controls (620,138.66) and (496.88,170) .. (345,170) .. controls (193.12,170) and (70,138.66) .. (70,100) -- cycle ;
\draw   (110,105) .. controls (110,85.67) and (170.14,70) .. (244.33,70) .. controls (318.52,70) and (378.67,85.67) .. (378.67,105) .. controls (378.67,124.33) and (318.52,140) .. (244.33,140) .. controls (170.14,140) and (110,124.33) .. (110,105) -- cycle ;
\draw   (290,105) .. controls (290,85.67) and (352.98,70) .. (430.67,70) .. controls (508.35,70) and (571.33,85.67) .. (571.33,105) .. controls (571.33,124.33) and (508.35,140) .. (430.67,140) .. controls (352.98,140) and (290,124.33) .. (290,105) -- cycle ;
\draw    (180,10) -- (180,190) ;
\draw    (420,10) -- (420,200) ;
\draw    (420,250) -- (420,290) ;
\draw    (70,200) -- (420,200) ;
\draw    (70,250) -- (420,250) ;
\draw    (70,200) .. controls (32.67,202) and (30.67,246.67) .. (70,250) ;
\draw    (70,190) -- (180,190) ;
\draw    (70,260) -- (180,260) ;
\draw    (180,260) -- (180,290) ;
\draw    (70,190) .. controls (12.67,189.33) and (12.67,263.33) .. (70,260) ;
\draw    (20,10) -- (20,180) ;
\draw    (630,10) -- (630,210) ;
\draw    (630,240) -- (630,290) ;
\draw    (70,240) -- (630,240) ;
\draw    (70,210) -- (630,210) ;
\draw    (70,210) .. controls (52,211.33) and (54.67,244) .. (70,240) ;
\draw    (20,270) -- (20,290) ;
\draw    (20,180) -- (110,180) ;
\draw    (20,270) -- (110,270) ;
\draw    (110,180) -- (110,270) ;

\draw (64,52.4) node [anchor=north west][inner sep=0.75pt]    {$S_{30}$};
\draw (141,110.4) node [anchor=north west][inner sep=0.75pt]    {$S_{31}$};
\draw (501,112.4) node [anchor=north west][inner sep=0.75pt]    {$S_{20}$};
\draw (621,292.4) node [anchor=north west][inner sep=0.75pt]    {$Y_{01}$};
\draw (22,293.4) node [anchor=north west][inner sep=0.75pt]    {$Y_{30}$};
\draw (152,292.4) node [anchor=north west][inner sep=0.75pt]    {$Y_{21}$};
\draw (221,292.4) node [anchor=north west][inner sep=0.75pt]    {$Y_{20}$};
\draw (461,292.4) node [anchor=north west][inner sep=0.75pt]    {$Y_{10}$};
\draw (401,292.4) node [anchor=north west][inner sep=0.75pt]    {$Y_{11}$};
\end{tikzpicture}
\centering
\caption{The nine hypersurfaces in Definition-Proposition~\ref{prop:hypersurfaces3}.}
\label{fig:prop:hypersurfaces3}
\end{figure}
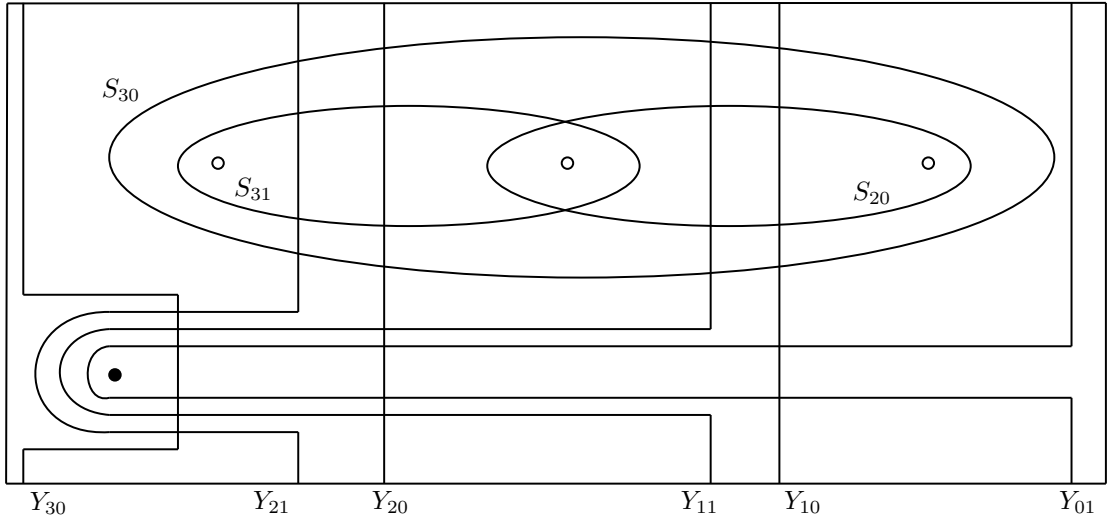

\subsection{Family of metrics}
Let $W_{v*,u*'}^o$ be the double cover of $ W_{v*,u*'}^o$ along $\Sigma_{v*,u*'}^o$, where $*,*' \in \{0,1\}$.
We will consider a family of metrics $P_{v*,u*'}$ over $W_{v*,u*'}^o$, invariant under deck transformation, such that $P_{v*,u*'}$ is a compact polytope.
This family is essentially the same as \cite{Bloom2011}, which can be chosen invariantly.
We will describe the orbifold metrics downstairs.
For the full combinatorial structures, see \cite[Section~2 and~5]{Bloom2011}.

\textbf{M1.} Assume $v = \bfone$ and $u = \bfzero$.
Suppose $( W^o,\Sigma^o)$ is the corresponding cobordism.
Choose an initial orbifold metric $\check g$ that is cylinder near all of the hypersurfaces $( Y_{w},K_w)$ from Proposition~\ref{prop:hypsurfaces2}, where $w \in  \bbZ^N \times \{0,1\}$.
Consider a sequence $w^i$ such that
\[
(1,\dots,1) > w^1 > w^2 > \dots > w^{N} > (0,\dots,0).
\]
We define a $[0,\infty)^N$-family of metrics by inserting necks to $\check g$ as
\[
 W_{\mathbf{1}1,w^1} 
\bigcup_{ Y_{w^1}}
([-T_1,T_1] \times  Y_{w^1})
\bigcup_{ Y_{w^1}}
 W_{w^1,w^2} 
\bigcup_{ Y_{w^2}}
([-T_2,T_2] \times  Y_{w^2})
\bigcup_{ Y_{w^2}}
\cdots
\bigcup_{ Y_{w^N}}
([-T_N,T_N] \times  Y_{w^N})
\bigcup_{ Y_{w^N}}
 W_{w^N,\mathbf{0}0},
\]
where orbifold loci were suppressed in our notations.
By adding broken metrics, we extend $[0,\infty)^N$ to a cube $[0,\infty]^N$.
There are $N!$ such cubes, and these cubes can be glued to a polytope $P_{v1,u0}$ along the codimenison-1 faces.
Indeed, a face corresponding to a sequence $w^1 > \dots > w^{N-1}$ which lies on exactly two $N$-cubes.

Similarly, suppose $v,u \in \bbZ^N$ are such that $v > u$ and $|v-u|_{\infty} \le 1$. 
Recall $( W_{v,u}^o,\Sigma_{v,u}^o)$ can be considered as the submanifold
$( W_{v1,u1}^o,\Sigma_{v0,u0}^o)$ of the cobordism from $\mathbf{1}$ to $\mathbf{0}$ (up to translation of the indices).
In the same way above, we construct a family of metrics on $( W_{v*,u*'}^o,\Sigma_{v*,u*'}^o)$, parmetrized by a  $(|v|_1-|u|_1 + |*-*'|- 1)$-dimensional polytope $P_{v*,u*'}$.
That is, we build the polytopes from cubes and in each cube stretch necks around the hypersurfaces $ Y_{w^i}$, where
\[
	(v,*) > w^1 > w^2 > \dots > w^{N'} > (u,*').
\]

\textbf{M2.}
Assume $v = \bftwo$ and $u = \bfzero$, and equip the corresponding $( W^o,\Sigma^o)$ with an initial orbifold metric $\check g$.
Suppose $\check g$ is cylindrical in some neighbourhoods of hypersurfaces in Proposition~\ref{prop:hypersurfaces2o}.
The polytope $P_{v,u}$ is built from $2N$-dimensional cubes $[0,\infty]^{2N}$.
Each cube corresponds to stretching $\check g$ along $2N$ mutually disjoint hypersurfaces.
In other words, each collection of such hypersurfaces will be of the form
\[
 Y_{w^1},\dots,  Y_{w^{N_Y}},  S_{k_1},\dots, S_{k_{N_S}},
\]
where $\mathbf{2}1 > w^1 > \dots > w^{N_Y} > \mathbf{0}0$ and the $k_i$-th coordinate of all $w^j$ are $0$ or $2$.
We obtain $P_{v,u}$ by gluing along  $(2N-1)$ dimensional faces.

By restricting to hypersurfaces in submanifolds, we obtain $P_{v*,u*'}$ for $v,u \in \bbZ^N$ such that $|v - u|_{\infty} \le 2$ and $v* > u*'$.
In particular, the dimension is $2m + (N-m) + |*-*'|_1 - 1$, where $m$ is the number of coordinate $i$ for which $|v_i - u_i| = 1$.

\textbf{M3.}
Assume there is a single crossing.
Let $( W^o,\Sigma^o)$ be the punctured cobordism from $( Y,K_3)$ to $( Y,K_0)$.
Start with an orbifold metric $( W^o,\Sigma^o)$ that is product-like near the $9$ hypersurfaces:
$( Y_{30}, K_{30})$, $( Y_{21},K_{21})$, $( Y_{20},K_{20})$, $( Y_{11},K_{11})$,  $( Y_{10},K_{10})$, $( Y_{01},K_{01})$, $( S_{3,1}, \U_1)$, $( S_{2,0},\U_1)$, and $( R_{3,0}, \U_2)$.
Here $\U_1$ denotes a unknot and $\U_2$ denotes a $2$-component unlink.
We stretch the three necks over a cube $[0,\infty]^3$ and glue the resulting cubes to form a $3$-dimensional polytope $P_{31,00}$, and two $2$-dimensional polytopes $P_{30,00}$ and $P_{31,01}$.
The $2$-dimensional polytopes are pentagons.
For example, $P_{31,01}$ has five edges, corresponding to metrics that are broken along the hypersurfaces $ Y_{21},  Y_{11}, S_{3,1},  S_{2,0},  R_{3,0}$, respectively.

\subsection{Identifying $E_{\infty}$}
Assume $v \ge u$ are in $\bbZ^N$. Define
\[
\bfC[v,u] = 
\bigoplus_{v \ge y \ge u} \widetilde{C}(K_y)
\]
and 
\[
\bfF[v,u] = \sum_{v \ge v' \ge u' \ge u} f_{v',u'}: \bfC \to \bfC
\] 
Given $v \ge u$ in $\bbZ^N$, where $P_{v,u}$ is well-defined.
Recall that $\tilde C(K_{v})=\check C(K_{v1}) \oplus \check C(K_{v0})$.
Then $f_{v,u}$ is of the form
\[
	\begin{bmatrix}
		f_{v1,u1} & f_{v0,u1}\\
		f_{v1,u_0} & f_{v0,u_0}
	\end{bmatrix},
\]
where $f_{v*,u*'}:\check C(K_{v*}) \to \check C(K_{v*'})$ and $f_{v0,u1} = 0$.
In what follows, the superscript $\star \to \star'$ of $f$ denotes the types $\star,\star' \in \{o,s,u\}$ of critical points in the sources and targets, respectively.
In the decomposition, $\check C(K_{w}) = C^o(K_{w}) \oplus C^s(K_{w})$, we write
\[
	f_{v*,u*'} =
	\begin{bmatrix}
	f_{v*,u*'}^{o \to o} 
	&	
	\sum_{v* \le w \le u*'} f_{w,u*'}^{u \to o}\bar f_{v*,w}^{s \to u}
	\\	
	f_{v*,u*'}^{o \to s} 
	&	\bar{f}_{v*,u*'}^{s \to s} + \sum_{v* \le w \le u*'}f_{w,u*'}^{u \to s} \bar f_{v*,w}^{s \to s}
	\end{bmatrix}.
\]
For arbitrary $v > u$ and $* = *$, the map $f_{v*,u*}^{\star \to \star'}$ is defined by the mod-$2$ counts of the parameterized moduli spaces
\[
	M_z^+(\fraka;  W_{v*,u*}, \Sigma_{v*,u*};\frakb)_{P_{v*,u*}},
\]
where $\fraka$ and $\frakb$ are critical points of types $\star$ and $\star'$, respectively.
The count is zero if the dimension of $M_z$ is nonzero.
Note that in this case $ W_{v*,u*} =  W_{v*,u*}^o$ has no puncture.
Otherwise, we define $f_{v1,u0}^{\star \to \star'}$ to be the count of 
\[
M_z(\fraka, \frakc_{-2};  W^o_{v1,u0}, \Sigma_{v1,u0};\frakb)_{P_{v1,u0}},
\]
where this notation means we view $( W^o_{v1,u0}, \Sigma^o_{v1,u0})$ as a $3$-ended cobordism from $( Y_{v1},K_{v1}) \sqcup ( S_o, \U_1)$ to $( Y_{u0},K_{u0})$, and insert the critical point $\frakc_{-2}$ on the $( S_o, \U_1)$ end.
Similarly, $\bar{f}_{v*,u*}^{\star \to \star'}$ is defined using the reducible moduli space 
\[
M_z^{\red}(\fraka;  W_{v*,u*}, \Sigma_{v*,u*};\frakb)_{P_{v*,u^*}}, \ \text{or }  M_z^{\red}(\fraka, \frakc_{-2};  W^o_{v1,u0}, \Sigma_{v1,u0};\frakb)_{P_{v1,u0}}.
\]
If $v=u$, then we replace $( W^o_{v,u},\Sigma^o)$ with the cylinder $([0,1] \times  Y \setminus  B_o, [0,1] \times K_v \setminus  B_o)$ in definitions of $f_{v*,u*}^{\star \to \star'}$ and $\bar f_{v*,u*}^{\star \to \star'}$ so that the moduli spaces $M_z$ and $M_z^{\red}$ are what show up in the definitions of the differentials of $\widetilde{\HMR}_*( Y,K_v)$ as in Section~\ref{sec:HMRtildedefn}.
\begin{lem}
\label{lem:10_complex}
Let $\bfC = \bfC[\mathbf{1},\mathbf{0}]$ and $\bfF=\bfF[\mathbf{1},\mathbf{0}]$. Then $\bfF^2 = 0$, i.e. $(\bfC,\bfF)$ is a chain complex.
\end{lem}
\begin{proof}
The proof makes use of the family of metrics in $\mathbf{M1}$. The arguments are identical to \cite[Lemma~4.7 and 8.2]{Bloom2011}.
One enumerates the boundaries of the $1$-dimensional compactified moduli spaces and the perspective reducible versions
\[
M_z^+(\fraka;  W_{v*,u*}, \Sigma_{v*,u*};\frakb)_{P_{v*,u^*}}, \quad
M_z^+(\fraka, \frakc_{-2};  W^o_{v1,u0}, \Sigma_{v1,u0};\frakb)_{P_{v1,u0}}.
\]
The resulting maps coincide with the entries of $\bfF^2$.
The lemma then follows from the fact that the number of boundary points of the parametrized moduli space is zero modulo two by \cite[Lemma~9.13 and Proposition~12.6]{ljk2022}.
\end{proof}
\begin{lem}
	\label{lem:20_complex}
	$(\bfC[\bftwo,\bfzero],\bfF[\bftwo,\bfzero])$ is a complex.
\end{lem}
\begin{proof}
The proof uses of the family of metrics in $\mathbf{M2}$.
The new phenomenon arises from the neck stretching along the hypersurfaces $ S_k$.
As noted in Propositon~\ref{prop:double_cover_triangle}, $ S_k$ bounds a $4$-ball in which the branch locus is a $\mathbb{RP}^2$ having self-intersection $(+2)$.
Moreover, the covering involution on the branched double cover is equivalent to the standard conjugation action on $\overline{\mathbb{CP}}^2 \setminus \text{int}(\textsf{Ball})$.
It follows that there is a fixed-point free symmetry on the set of monopoles over the facet corresponding to the metrics broken along $ S_k$.
Modulo two, this facet contributes trivially.
The techniques of \cite[Lemma~8.2]{Bloom2011} can be adapted to conclude $\bfF[\bftwo,\bfzero]^2=0$.
See the discussion following \cite[Remark~8.3]{Bloom2011}.
\end{proof}
The next proposition is the $\widetilde{\HMR}_{\bullet}$ version of unoriented skein exact triangle.
Assume $K_2,K_1,K_0$ are related by a single skein relation, where a based point $\infty$ is fixed.
\begin{prop}
\label{prop:triangle}
There is a quasi-isomorphism from the chain complex $\widetilde C(K_{2})$ to the mapping cone of the cobordism map $f_{1,0}:\widetilde C(K_{1}) \to \widetilde C(K_{0})$, induced by $\Sigma_{1,0}: K_1 \to K_0$.
In homology, there is a $3$-periodic long exact sequence
\[
\begin{tikzcd}
	\cdots \ar[r,"f_{4,3}"] 
	&
	\widetilde{\HMR}_{\bullet}(K_3) \ar[rr,"f_{3,2}=f_{0,-1}"] 
	&
	&
	\widetilde{\HMR}_{\bullet}(K_2) \ar[r,"f_{2,1}"] 
	&
	\widetilde{\HMR}_{\bullet}(K_1) \ar[r,"f_{1,0}"] 
	&
	\widetilde{\HMR}_{\bullet}(K_0)
	\ar[r,"f_{0,-1}"] 
	&
	\cdots
\end{tikzcd}.
\]
\end{prop}
\begin{proof}
Just like \cite[Theorem~2.4]{KMOS2007}, the proof is based on an algebraic lemma, where it suffices to verify $f_{i+1,i} \circ f_{i,i-1}$ is homotopically trivial by a chain homotopy $H_{i+2,i}:\tilde C(K_{i+2}) \to \tilde C(K_{i})$, and that $f_{i+1,i} \circ H_{i+3,i+1} + H_{i+2,i} \circ f_{i+3,i+2}$ is a quasi-isomorphism.
The first assertion is a consequence of Lemma~\ref{lem:20_complex}, where $H_{i+2,i}$ is precisely $f_{i+2,i}$.
Indeed, the lemma reads
\[
	f_{i,i} \circ f_{i+2,i} + f_{i+2,i} \circ f_{i+2,i+2} = f_{i+1,i} \circ f_{i+2,i+1},
\]
where $f_{i,i}$ is the differential of $\tilde C(K_{i})$.
To see the second assertion, we consider the $P_{31,00}$ family metrics on $( W_{31,00}^o,\Sigma_{31,00}^o)$ in \textbf{M3}.
The crucial topological fact, as in Proposition~\ref{prop:double_cover_triangle}, is that once the metric is broken along the hypersurface $( R_{3,0},\U_2)$, the cobordism is a union of $( B_{3,0},M_{3,0})$ and a piece which is simply the cylinder $[0,3] \times ( Y,K_3)$ with a ball removed.
It is the $ R_{3,0}$ facet which contributes the quasi-isomorphism to the second assertion.
To this end, let $R \cong [-\infty,+\infty]$ be the subfamily of metrics in $P_{31,00}$ which stretches the hypersurfaces $ S_{3,1}$ and $ S_{2,0}$
Let $\bar n_u \in C^u( R_{3,0},M_{3,0})$ count the monopoles over $R$ which limits to unstable critical points.
We assume appropriate metrics and perturbations on $( R_{3,0},M_{3,0})$ and the cobordisms are chosen, as in  \cite[Section~5]{ljk2023triangle}.
Roughly speaking, this is to ensure all relevant Seiberg-Witten trajectories are reducible.

We define an operator $L_{3,0}:\tilde C(K_3) \to \tilde C(K_0)$ as follows.
Let $L_{31,01}^{* \to *} = f_{31,01}^{u* \to *}( \bar n_u \otimes  \cdot)$ and $\bar L_{31,01}^{* \to *} = \bar f_{31,01}^{u* \to *}(\ \bar n_u \otimes  \cdot)$,
where we insert $\bar n_u$ to the $ R_{3,0}$ end.
Recall the definition of $f_{31,01}$ involves the $1$-dimensional family of metrics $P_{31,01}$.
This is the same setting as the exact triangle in $\HMR^{\circ}_{\bullet}$. 
By \cite[Proposition~5.3]{ljk2023triangle}, we deduce that
\[
f_{01,01}f_{31,01} + f_{31,01}f_{31,31} = 
f_{21,01}f_{31,21} + f_{11,01}f_{31,11} +
L_{31,01}.
\]
Similarly, the operator $L_{30,00}$ concerns only a $2$-dimensional face of $P_{31,00}$ and the analogous identity in the `to'-flavoured triangle is
\[
f_{00,00}f_{30,00} + f_{30,00}f_{30,30} = 
f_{20,00}f_{30,20} + f_{10,00}f_{30,10} +
L_{31,01}.
\]

The off-diagonal entry $L_{31,00}$ of $L$ requires the full $3$-dimensional family $P_{31,00}$.
Let $L_{31,00}^{* \to *} = f_{31,00}^{uu* \to *}(\frakc_{-2} \otimes \ \bar n_u \otimes  \cdot)$ and $\bar L_{31,00}^{* \to *} = \bar f_{31,00}^{uu* \to *}(\frakc_{-2} \otimes \ \bar n_u \otimes  \cdot)$, where $\frakc_{-2}$ is inserted into the boundary $( S_o, \U_1)$.
Now, $L_{31,00}$ counts monopoles over the two dimensional family of metrics broken along $ R_{3,0}$ and stretches three hypersurfaces.
Bloom's argument \cite[Equation~(31)]{Bloom2011} goes over verbatim:
\[
f_{31,31}L_{31,00} + L_{31,00}f_{00,00}
=
L_{30,00}f_{31,30} + f_{01,00}L_{31,01}
\]
Here, $L_{31,10}$ and $L_{21,00}$ are defined analogously by counting monopoles over $1$-dimensional families.
Putting these together,
\[
f_{0,0}f_{3,0} + f_{3,0}f_{3,3} = 
f_{2,0}f_{3,2} + f_{1,0}f_{3,1} +
L_{3,0}.
\]
This identity amounts to the nine-term identities involving $2$-by-$2$ matrices in \cite[Section~8.0]{Bloom2011}.
The quasi-isomorphism assertion is a consequence of the quasi-isomorphism in the 'to'-version of exact triangle, as the diagonal entries of
\[
	\begin{bmatrix}
		L_{31,01} & 0\\
		L_{31,00} & L_{30,00}
	\end{bmatrix}
\]
are shown in \cite[Lemma~5.9]{ljk2023triangle} to be isomorphisms. 
\end{proof}
We have all the pieces to prove Theorem~\ref{thm:quasi_iso_1}.
\begin{proof}[Proof of Theorem~\ref{thm:quasi_iso_1}]
By Proposition~\ref{prop:triangle}, there is a chain of isomorphisms whose final target is the complex $(\tilde C(K_{\bftwo}),f_{\bftwo,\bftwo})$:
\[
	\bfC[\bfone,\bfzero] \rightarrow 
	\bfC[(1,\dots,1,2),(0,\dots,0,2)] \to 
	\bfC[(1,\dots,1,2,2),(0,\dots,0,2,2)] \to \cdots \to
	\bfC[\bftwo,\bftwo].
\]
Note that each arrow is a chain map, defined using differentials between the vertex $(2,\dots,2,2,0, \dots,1)$ and $(2,\dots,2,0,0,\dots,0)$.
\end{proof}
\subsection{Proof of Theorem~\ref{thm:SS_intro}}
To prove Theorem~\ref{thm:SS_intro}, recall 
\[
	\bfC = 
	\bigoplus_{v \in \{0,1\}^N} \widetilde{C}(K_v)
\]
with differential 
\[\bfF = \sum_{v \ge u} f_{v,u}: \bfC \to \bfC\] 
having components
\[f_{v,u}:\widetilde{C}(K_v) \to \widetilde{C}(K_u).
\]

Recall also from the beginning of the section that $D_v$ is the diagram of $K_v$ consisting of a disjoint union of circles.
Let $\overline{D}_v$ be the diagram of the $v$-resolution of the mirror link $K_v$.
Taking mirrors reverses the roles of $0$- and $1$-resolutions.
As a result, the Khovanov differential is now formed by maps $d_{v,v'}$ for edges $v >_1 v'$ in the cube, i.e.
\[
d_{v,v'}:\CKh(\overline{D}_v) \to \CKh(\overline{D}_{v'}).
\]

We study the $E_1$-page. 
By \cite[Prop.~14.3]{ljk2022}, $\widecheck{\HMR}_*(K_v)$ is a $\bbF_2[\upsilon_{\infty}]$ module over the cohomology of the Picard torus of $\Sigma_2(S^3,K_v)$.
The latter cohomology ring in turn can be identified with $\Lambda^*(H_1(\Sigma_2(S^3,K_v);\bbF_2))$.
Since $\widetilde{\HMR}_{\bullet}(K_v)$ can be identified with $\ker(\upsilon_{\infty})$ in $\widecheck{\HMR}_{\bullet}(K_v)$, the Floer homology $\widetilde{\HMR}_{\bullet}(K_v)$ is a rank-$1$ free module over $\Lambda^*V$.
In particular, 
\[\widetilde{\HMR}_{\bullet}(K_v) \cong \CKhr(\overline{D}_v).\]
The above isomorphisms can be made more explicit.
Fix $v \in \{0,1\}^N$ and let $x_o$ be the component of $\overline{D}_v$ that contains $\infty$.
For each circle $x \ne x_o$ in $D_v$, choose an arc $\eta_x$ in $S^3$ connecting the $x_o$- and $x$-components of $K_v$.
The lift of $\eta_x$ in the double branched cover represents a class in $H_1(\dbcv_2(S^3,K_v))$.
The assignment of $x \mapsto \eta_x$ provides an isomorphism
\[
V(\overline{D}_v)/(x_o) \to H_1(\dbcv_2(S^3,K_v);\bbF_2)
\]
and extends to 
\[
	\psi_v:\Lambda^*(V(\overline{D}_v)/(x_o))  \xrightarrow{\cong} \Lambda^*H_1(\dbcv_2(S^3,K_v);\bbF_2).
\]
To match the $E_1$-page with the reduced Khovanov complex, we need to prove the following diagram commutes for $v >_1 v'$:
\[
\begin{tikzcd}
\Lambda^* V(\overline{D}_v)/(x_o) \ar[r,"d_{v,v'}"] \ar[d,"\psi_v"] & \Lambda^* V(\overline{D}_{v'})/(x_o) \ar[d,"\psi_{v'}"]\\
H_1(\dbcv_2(S^3,K_v);\bbF_2) \ar[r,"f_{v,v'}"] & H_1(\dbcv_2(S^3,K_{v'});\bbF_2).
\end{tikzcd}
\]
Given $v > v'$,
The cobordism $\Sigma_{v,v'}$ is either the pair-of-pants cobordism that merges two circles in $D_v$, or the reverse cobordism that splits a circle into two.
The following computation of $f_{v,v'}$ is analogous to \cite[Cor.~9.2]{Bloom2011}.
\begin{prop}
	\label{prop:formula_pair_pants}
	Let $K$ be an $n$-component unlink.
	The homology $\widetilde{\HMR}_{\bullet}(K)$ is a rank-$1$ free module over the exterior algebra $\Lambda^*V$, where $V$ is an $(n-1)$-dimensional vector space.
	Let $\eta$ be an arc whose end points are on two components of $K$, such that the lift of $\eta$ in $\dbcv_2(S^3,K)\cong \#_n(S^1 \times S^2)$ generates the homology of one of the $S^1$-factor.
	Let $\Sigma_{merge}: K \to K'$ be the $2$-dimensional cobordism corresponding adding $1$-handle along $\eta$.
	Then $\widetilde{\HMR}_{\bullet}(K')$ is a rank-$1$ free module over $\Lambda^*V'$, where $V' \cong V/(\eta)$ is an $(n-2)$-dimensional moduli space.
	Denote by $\pi: V \to V'$ the corresponding projection. Then
	the cobordism map is given by, for any $\xi \in \Lambda^*(V)$,
	\[
	\widetilde{\HMR}(\Sigma_{merge})(\xi \cdot \Theta) = \pi(\xi) \cdot \Theta',
	\]
	where $\Theta,\Theta'$ are the respective generators of $\widetilde{\HMR}_{\bullet}(K)$ and $
	\widetilde{\HMR}_{\bullet}(K')$.
	Next,
	let $\Sigma_{split}:K \to K''$ be a split cobordism.
	There is an $n$-dimensional vector space $V''$ and an injection $i:V \to V''$.
	Moreover, there exists $x'' \in V''$ such that $V \cong V''/\langle x'' \rangle$ and $\Lambda^*V$ embeds into $ \Lambda^*V''$ via $\xi \mapsto \xi \wedge x''$.
	Then $\widetilde{\HMR}_{\bullet}(K'')$ is a rank-$1$ free module over $\Lambda^*V''$, and
	\[
	\widetilde{\HMR}(\Sigma_{split})(\xi \cdot \Theta) = (\xi \wedge x'') \cdot \Theta''.
	\]
\end{prop}
\begin{proof}
The first two statements were explained in the preceding discussion.
In parictular, $V = H_1(\Sigma_2(S^3,K))$. 
Let $x \in V$ be the homology class representing $\eta$ in the double branched cover.
Then $x = 0$ in $H_1(\dbcv_2([0,1] \times S^3,\Sigma_{merge}))$.
Since cobordism maps respects the $H_1$-action, $(x \wedge \xi) \cdot \Theta$ is in the kernel of $\widetilde{\HMR}(\Sigma_{merge})$.
On the other hand, $\Sigma_{merge}$ can be viewed as a skein cobordism $\Sigma_{1,0}:K_1 \to K_0$ as in Figure~\ref{fig:merge_skein_cob}.
\begin{figure}
\begin{center}
\tikzset{every picture/.style={line width=0.5pt}} 
\begin{tikzpicture}[x=0.5pt,y=0.5pt,yscale=-1,xscale=1]
	\draw    (63.33,14.43) -- (143.17,86.08) ;
	\draw    (101.11,59.11) -- (71.09,84.96) ;
	\draw    (306.18,16.67) .. controls (372.71,45.77) and (332.79,88.32) .. (307.29,96.15) ;
	\draw    (400.43,14.43) .. controls (339.17,34.17) and (367.17,86.67) .. (400.43,97.27) ;
	\draw    (541.26,14.43) .. controls (583.4,40.18) and (572.67,54.17) .. (629.97,15.55) ;
	\draw    (549.02,93.91) .. controls (580.17,79.17) and (608.9,58.09) .. (641.06,90.55) ;
	\draw    (63.33,14.43) .. controls (24,-14.83) and (17,132.17) .. (71.09,84.96) ;
	\draw    (147.78,15.11) .. controls (200.83,6.03) and (191.78,121.78) .. (143.17,86.08) ;
	\draw    (306.18,16.67) .. controls (251.84,-2.37) and (247.41,115.18) .. (307.29,96.15) ;
	\draw    (400.43,14.43) .. controls (422.38,6.87) and (436.7,26.73) .. (440.36,49.13) .. controls (444.84,76.53) and (433.38,107.74) .. (400.43,97.27) ;
	\draw    (541.26,14.43) .. controls (486.93,-4.61) and (489.14,112.95) .. (549.02,93.91) ;
	\draw    (629.97,15.55) .. controls (646.17,8.17) and (671.78,22.25) .. (675.44,44.65) .. controls (679.11,67.06) and (664.35,109.59) .. (641.06,90.55) ;
	\draw    (112,47) .. controls (122.44,33.78) and (131.11,17.78) .. (147.78,15.11) ;
	
	\draw (90.53,113.57) node [anchor=north west][inner sep=0.75pt]    {$K_{2}$};
	\draw (340.06,113.93) node [anchor=north west][inner sep=0.75pt]    {$K_{1}$};
	\draw (594.01,113.06) node [anchor=north west][inner sep=0.75pt]    {$K_{0}$};
\end{tikzpicture}
\end{center}
\caption{A merge cobordism $\Sigma_{merge}$ as a skein cobordism $\Sigma_{1,0}$.}
\label{fig:merge_skein_cob}
\end{figure}
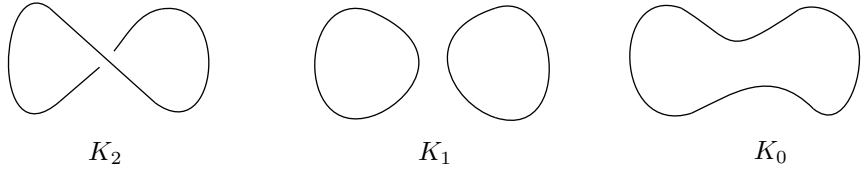
Since $K_2$ is isotopic to $K_0$ the triangle splits for rank reason:
\[
\begin{tikzcd}
	0 \ar[r] & \widetilde{\HMR}_{\bullet}(K_2) \ar[r] &\widetilde{\HMR}_{\bullet}(K_1) \ar[rr,"\widetilde{\HMR}(\Sigma_{merge})"]
	& {} &\widetilde{\HMR}_{\bullet}(K_0) \ar[r] &
	0.
\end{tikzcd}
\]
In particular, $\widetilde{\HMR}(\Sigma_{merge})= \widetilde{\HMR}(\Sigma_{1,0})$ is surjective.
Similarly, $\widetilde{\HMR}(\Sigma_{split}) = \widetilde{\HMR}(\Sigma_{2,1}) $ is injective.

Let $V'' = H_1(\Sigma_2(S^3,K''))$ and $x'' \in V''$ be represented by the arc joining the two output circles in $K''$ from splitting.
In particular, there is an injection $i:V \to V''$ and $V''/\langle x''\rangle \cong V$.
Since $x''$ is null-homologous in $\Sigma_2(I \times S^3,\Sigma_{split})$, 
\[
	x'' \cdot \widetilde{\HMR}(\Sigma_{split})(\xi\cdot \Theta) =\widetilde{\HMR}(\Sigma_{split})((0 \wedge \xi)\cdot \Theta) = 0.
\]
So the image of $\widetilde{\HMR}(\Sigma_{split})$ lies in $x''\wedge \widetilde{\HMR}_{\bullet}(K'')$.
\end{proof}
It follows from Propositon~\ref{prop:formula_pair_pants} that the component $f_{v,v'}$ of the differential $\bfF_1$ on the $E_1$-page can be identified with 
\[
d_{u,v}: \CKhr(\overline{D}_u) \to \CKhr(\overline{D}_v)
\]
via the isomorphisms $(\psi_v: v \in \{0,1\}^N)$.
In effect, we have proven that the following fact.
\begin{prop}
	\label{prop:E1diso}
	There is an isomorphism of chain complexes from $(\CKhr(\overline{D}),d)$ to $(E_1(\bfC(K)),\bfF_1)$. 
	In other words, $E_2 \cong \Khr(\overline{K})$. \hfill \qedsymbol
\end{prop}
The final component of Theorem~\ref{thm:SS_intro} is the invariance of the complex $\bfC(D)$ with respect to link isotopy and the auxiliary data.
To this end, recall that $\HMR^{\circ}_{\bullet}(K_v)$ depends on the choice of an $\upiota$-invariant metric, real spin\textsuperscript{c} structures, and regular perturbations.
Denote a choice of such data by $\frakd$ and denote the corresponding complex by $\bfC^{\frakd}(D)$. 
\begin{prop}
Given two $D,D'$ based diagrams of $(K,p)$, and two set of auxiliary data of $\frakd,\frakd'$.
Then there is a filtered (with respect to vertex weight) chain map $\phi:\bfC^{\frakd}(D) \to \bfC^{\frakd'}(D')$ inducing isomorphism on the $E^i$ page for all $i \ge 2$.
\end{prop}
\begin{proof}
The proof when $D = D'$ is a straightforward adaptation of \cite[Theorem~7.4]{Bloom2011}.
We outline it below.
Consider the family of $3$-manifolds $\{Y_v: v \in \{0,1\}^{N+1}\}$ in Definition-Proposition~\ref{prop:hypsurfaces1}.
Double this family to $\{Y_{v*} : v \in \{0,1\}^{N+2}\}$ by taking slight left translates.
We will produce families  of metrics $\{\underline{P_{v,u}}:u \ge v\}$ to set up the following complex
\[
	\underline{\bfC}=
	\bigoplus_{v*\in \{0,1\}^{N+1}} \widetilde{C}(K_v),
\]
equipped with a differential $\underline{\bfF}$ formed by components $\underline{f}_{v*,u*'}:\widetilde{C}(K_v) \to \widetilde{C}(K_u)$.
Choose metrics over $\underline{P_{v,u}}$ and perturbations so that $\bfC^{\frakd}(D)$ can be viewed as a complex over $\{0,1\}^N \times \{1\}$, obtained as the quotient complex of $\underline{\bfC}$ by the subcomplex $\bfC^{\frakd'}(D')$ over $\{0,1\}^N \times \{0\}$.
Moreover, one can show $\underline{\bfF}^2 = 0$ and the isomorphism in the statement is given by
\[
\phi = \sum_{v \ge u} \underline{f}_{v1,u0}.
\]
For the case $D \ne D'$, it suffices to assume $D$ and $D'$ are related by a single Reidemeister move.
The map $\phi$ in the Heegaard-Floer case is defined by Baldwin~\cite{Baldwin2011}. 
His construction applies readily to both the ordinary monopole Floer homology of the double cover and our setting.
Moreover, his proof of isomorphism properties of $\phi$ does not involve analysis of Floer theories, and hence works verbatim for $\widetilde{\HMR}$.
\end{proof}

\section*{Acknowledgement}
The author would like to thank Ciprian Manolescu for email correspondence which motivated this work, and  Masaki Taniguchi for several comments on the draft.
The author is also grateful to his advisor Peter Kronheimer for many helpful discussions during this project.

\bibliographystyle{alpha}
\bibliography{hmr_tilde.bib}
\Addresses

\end{document}